\documentclass[12pt,a4paper]{amsart}  
\usepackage{underscore}
\usepackage{amsmath,amsthm,amssymb,amsfonts,graphicx}
\usepackage[T1]{fontenc}
\usepackage[utf8]{inputenc}
\usepackage{stmaryrd,float}
\usepackage{paralist,enumerate}

\usepackage[numbers,sort&compress]{natbib}
\makeatletter
\def\NAT@spacechar{~}
\makeatother

\makeatletter
\def\thm@space@setup{
\thm@preskip=4mm
\thm@postskip=0mm
}
\makeatother

\usepackage[usenames,dvipsnames,svgnames,table]{xcolor}
\usepackage{cmbright}

\usepackage[
pdftitle = {Subgraph Densities in a Surface},
pdfauthor = {Huynh -- Joret -- Wood},
colorlinks =true,
linkcolor={black},
urlcolor={Navy},
filecolor={Navy},
citecolor={black},
anchorcolor={Navy},
menucolor=Navy]{hyperref} 

\usepackage[noabbrev,capitalise]{cleveref}
\crefname{lemma}{Lemma}{Lemmas}
\crefname{theorem}{Theorem}{Theorems}
\crefname{proposition}{Proposition}{Propositions}

\newcommand{\defn}[1]{\textcolor{Maroon}{\emph{#1}}}

\usepackage{etoolbox}
\patchcmd{\section}{\scshape}{\bfseries}{}{}
\makeatletter
\renewcommand{\@secnumfont}{\bfseries}
\makeatother

\newcommand{\ceil}[1]{\lceil{#1}\rceil}
\newcommand{\floor}[1]{\lfloor{#1}\rfloor}

\renewcommand{\ge}{\geqslant}

\renewcommand{\geq}{\geqslant}
\renewcommand{\leq}{\leqslant}

\usepackage{aliascnt}

\newtheorem{theorem}{Theorem}[section]

\newaliascnt{lemma}{theorem}
\newtheorem{lemma}[lemma]{Lemma}
\aliascntresetthe{lemma}

\newaliascnt{corollary}{theorem}
\newtheorem{corollary}[corollary]{Corollary}
\aliascntresetthe{corollary}

\newaliascnt{conjecture}{theorem}
\newtheorem{conjecture}[conjecture]{Conjecture}
\aliascntresetthe{conjecture}

\newaliascnt{open}{theorem}
\newtheorem{open}[open]{Open Problem}
\aliascntresetthe{open}

\newaliascnt{claim}{theorem}
\newtheorem{claim}[claim]{Claim}
\aliascntresetthe{claim}

\DeclareMathOperator{\bd}{bd}
\DeclareMathOperator{\sep}{sep}

\begin{document}

\title{Subgraph Densities in a Surface}

\author[T.~Huynh]{Tony Huynh}
\author[G.~Joret]{Gwena\"el Joret}
\author[D.R.~Wood]{David R. Wood}
\address[T.~Huynh and D.R.~Wood]{\newline School of Mathematics
\newline Monash University
\newline Melbourne, Australia}
\email{\{tony.bourbaki@gmail.com, david.wood@monash.edu\}}

\address[G.~Joret]{\newline D\'epartement d'Informatique
\newline Universit\'e libre de Bruxelles
\newline Brussels, Belgium}
\email{gjoret@ulb.ac.be}

\thanks{All three authors are supported by the Australian Research Council. G.\ Joret is supported by an ARC grant from the Wallonia-Brussels Federation of Belgium and a CDR grant from the National Fund for Scientific Research (FNRS)}

\date{\today}
\sloppy

\begin{abstract}
Given a fixed graph $H$ that embeds in a surface $\Sigma$, what is the maximum number of copies of $H$ in an $n$-vertex graph $G$ that embeds in $\Sigma$? We show that the answer is $\Theta(n^{f(H)})$, where $f(H)$ is a graph invariant called the `flap-number' of $H$, which is independent of $\Sigma$. This simultaneously answers two open problems posed by Eppstein (1993).  The same proof also answers the question for minor-closed classes.  That is, if $H$ is a $K_{3,t}$ minor-free graph, then the maximum number of copies of $H$ in an $n$-vertex $K_{3,t}$ minor-free graph $G$ is $\Theta(n^{f'(H)})$, where $f'(H)$ is a graph invariant closely related to the flap-number of $H$.
Finally, when $H$ is a complete graph we give more precise answers. 
\end{abstract}

\maketitle 

%\tableofcontents

\section{Introduction}
 
 All graphs in this paper are undirected, finite, and simple, unless stated otherwise.  Many classical theorems in extremal graph theory concern the maximum number of copies of a fixed graph $H$ in an $n$-vertex graph in some class $\mathcal{G}$. Here, a \defn{copy} means a subgraph isomorphic to $H$. For example, Tur\'an's Theorem determines the maximum number of copies of $K_2$ (that is, edges) in an $n$-vertex $K_t$-free graph~\citep{Turan41}. More generally, Zykov's Theorem determines the maximum number of copies of a given complete graph $K_s$ in an $n$-vertex $K_t$-free graph~\citep{Zykov49}. The excluded graph need not be complete. The Erd\H{o}s--Stone Theorem~\citep{ES46}  determines, for every non-bipartite graph $X$, the asymptotic maximum number of copies of $K_2$ in an $n$-vertex graph with no $X$-subgraph. Analogues of the Erd\H{o}s--Stone Theorem for copies of $K_s$ have recently been studied by \citet*{AS19,AS16}. See~\citep{MaQiu20,AKS18,Timmons19,GMV19,GP19,GSTZ19,EMSG19,GKPP18,Luo18,NesOss11,GNV20} for recent related results.
 
This paper studies similar questions when the class $\mathcal{G}$ consists of the graphs that embed\footnote{See~\citep{MoharThom} for background about graphs embedded in surfaces. For $h\geq 0$, let $\mathbb{S}_h$ be the sphere with $h$ handles. For  $c\geq 0$, let $\mathbb{N}_c$ be the sphere with $c$ cross-caps. Every surface is homeomorphic to $\mathbb{S}_h$ or $\mathbb{N}_c$. The \defn{Euler genus} of $\mathbb{S}_h$ is $2h$. The \defn{Euler genus} of $\mathbb{N}_c$ is $c$. The \defn{Euler genus} of a graph $G$ is the minimum Euler genus of a surface in which $G$ embeds with no crossings. A graph $H$ is a \defn{minor} of a graph $G$ if a graph isomorphic to $H$ can be obtained from a subgraph of $G$ by contracting edges. If $G$ embeds in a surface $\Sigma$, then every minor of $G$ also embeds in $\Sigma$. } in a given surface $\Sigma$ (rather than being defined by an excluded subgraph). For graphs $H$ and $G$, let $C(H,G)$ be the number of copies of $H$ in $G$. For a  surface $\Sigma$, let $C(H,\Sigma,n)$ be the maximum of $C(H,G)$, where the maximum is taken over all $n$-vertex graphs $G$ that embeds in $\Sigma$. This paper determines the asymptotic behaviour of $C(H,\Sigma,n)$ as $n\rightarrow\infty$ for any fixed surface $\Sigma$ and any fixed graph $H$. % (which we assume is non-empty). 

Before stating our theorem, we mention some related results that determine $C(H,\mathbb{S}_0,n)$ for specific planar graphs $H$ where the surface is the sphere $\mathbb{S}_0$.  \citet*{AC84} determined  $C(H,\mathbb{S}_0,n)$ precisely if $H$ is either a complete bipartite graph or a triangulation without non-facial triangles. \citet*{HS79} studied $C(C_k,\mathbb{S}_0,n)$ where $C_k$ is the $k$-vertex cycle; they proved that $C(C_3,\mathbb{S}_0,n)=3n-8$ and $C(C_4,\mathbb{S}_0,n)=\frac12(n^2+3n-22)$. See~\citep{HS82,HHS01} for more results on  $C(C_3,\mathbb{S}_0,n)$ and see~\citep{Alameddine80} for more results on  $C(C_4,\mathbb{S}_0,n)$. \citet*{GPSTZa} proved that  $C(C_5,\mathbb{S}_0,n)=2n^2-10n+12$ (except for $n\in\{5,7\}$). If $P_k$ is the $k$-vertex path, then 
$C(P_4,\mathbb{S}_0,n)=7n^2-32n+27$ with finitely many exceptions~\citep{GPSTZb} and
$C(P_5,\mathbb{S}_0,n)=n^3+O(n^2)$~\citep{GGMPSXZ21}. \citet*{AC84} and independently \citet*{Wood-GC07} proved that $C(K_4,\mathbb{S}_0,n)=n-3$. More generally, Perles (see \citep{AC84}) conjectured that if $H$ is a fixed 3-connected planar graph, then $C(H,\mathbb{S}_0,n) = O(n)$. Perles noted the converse: If $H$ is planar, not 3-connected and $|V(H)|\geq 4$, then $C(H,\mathbb{S}_0,n) \geq \Omega(n^2)$.  Perles' conjecture was proved by \citet*{Wormald86} and independently by \citet*{Eppstein93}, who asked the following two open problems: 
\begin{itemize}
    \item Characterise the subgraphs occurring $O(n)$ times in graphs of given genus. 
    \item Characterise the subgraphs occurring a number of times which is a nonlinear function of $n$.
\end{itemize}
This paper answers both these questions (and more).

%PAPERS TO CITE \citep{LSS19a,LSS19b,Dowden16,LS19} ?

%\citet*{DFJSW} studied the maximum (total) number of cliques in an $n$-vertex graph that embeds in a given surface. 

%This paper studies extremal questions for graphs embedded in surfaces\footnote{See~\citep{MoharThom} for background about graphs embedded in surfaces. For $h\geq 0$, let $\mathbb{S}_h$ be the sphere with $h$ handles. For  $c\geq 0$, let $\mathbb{N}_c$ be the sphere with $c$ cross-caps. Every surface is homeomorphic to $\mathbb{S}_h$ or $\mathbb{N}_c$. The \defn{Euler genus} of $\mathbb{S}_h$ is $2h$. The \defn{Euler genus} of $\mathbb{N}_c$ is $c$. Euler's formula implies that every graph with $n\geq ????$ vertices that embeds in a surface of Euler genus $g$ has at most $3(n+g-2)$ edges. If, in addition, $G$ is bipartite, then $G$ has at most $2(n+g-2)$ edges.}. For a graphs $H$ and $G$, let $f(H,G)$ be the number of copies of $H$ in $G$. Here a \defn{copy} means a subgraph of $G$ isomorphic to $H$. For a graph $H$ and surface $\Sigma$ in which $H$ embeds, let $f(H,\Sigma,n)$ be the maximum of $f(H,G)$ where $G$ is an $n$-vertex graph embeddable in $\Sigma$. So $f(H,\Sigma,n)$ is the maximum number of copies of $H$ in an $n$-vertex graph that embeds in $\Sigma$. This paper studies the asymptotic behaviour of  $f$ as $n\rightarrow\infty$ with $\Sigma$ and $H$ fixed.

We start with the following natural question: when is $C(H,\Sigma,n)$ bounded by a constant depending only on $H$ and $\Sigma$ (and independent of $n$)? We prove that $H$ being 3-connected and non-planar is a sufficient condition. In fact we prove a stronger result that completely answers the question. We need the following standard definitions. A \defn{$k$-separation} of a graph $H$ is a pair $(H_1, H_2)$ of edge-disjoint subgraphs of $H$ such that $H_1 \cup H_2=H$, $V(H_1) \setminus V(H_2) \neq \emptyset$, $V(H_2) \setminus V(H_1) \neq \emptyset$, and $|V(H_1 \cap H_2)|=k$. A $k'$-separation for some $k'\leq k$ is called a \defn{$(\leq k)$-separation}. If $(H_1,H_2)$ is a separation of $H$ with $X=V(H_1)\cap V(H_2)$, then let $H_i^-$ and $H_i^+$ be the simple graphs obtained from $H_i$ by removing and adding all edges between vertices in $X$, respectively.  

A graph $H$ is \defn{strongly non-planar} if $H$ is non-planar and for every $(\leq 2)$-separation $(H_1, H_2)$ of $H$, both $H_1^+$ and $H_2^+$ are non-planar.  Note that every 3-connected non-planar graph is strongly non-planar.  The following is our first contribution. It says that $C(H,\Sigma,n)$ is bounded if and only if $H$ is strongly non-planar.

\begin{theorem}
\label{StronglyNonPlanar}
There exists a function $c_{\ref{StronglyNonPlanar}}(h, g)$ such that for every strongly non-planar graph $H$ with $h$ vertices and every surface $\Sigma$ of Euler genus $g$, 
\begin{equation*}
    C(H, \Sigma, n) \leq c_{\ref{StronglyNonPlanar}}(h, g).
\end{equation*}
Conversely, for every graph $H$ that is not strongly non-planar and for every surface $\Sigma$ in which $H$ embeds, there is a constant $c>0$ such that for all $n\geq 4|V(H)|$, there is an $n$-vertex  graph that embeds in $\Sigma$ and contains at least $cn$ copies of $H$; that is, $C(H,\Sigma,n)\geq cn$. 
\end{theorem}

There are two important observations about \cref{StronglyNonPlanar}. First, the characterisation of graphs $H$ does not depend on the surface $\Sigma$. Indeed, the only dependence on $\Sigma$ is in the constants. Second, \cref{StronglyNonPlanar} shows that $C(H,\Sigma,n)$ is either bounded or $\Omega(n)$. 

\cref{StronglyNonPlanar} is in fact a special case of the following more general theorem. The next definition is a key to describing our results. A \defn{flap} in a graph $H$ is a $(\leq 2)$-separation $(A,B)$ such that $A^+$ is planar.  Separations $(A,B)$ and $(C,D)$ of  $H$ are \defn{independent} if $E(A^-) \cap E(C^-) = \emptyset$ and $(V(A) \setminus V(B)) \cap (V(C) \setminus V(D))=\emptyset$.\footnote{It is worth noticing that neither condition implies the other. If $G$ is a $4$-cycle $abcd$, then the two $2$-separations $(A,B)$ and $(C,D)$ obtained by considering respectively the cutsets $\{b,d\}$ and $\{a,c\}$ satisfy the second condition but not the first. 
If $G$ consists of $5$ non-adjacent vertices $a, b, c, d, e$ then the two $2$-separations $(A,B)$ and $(C,D)$ with $V(A)=\{a, b, c, e\}$, $V(B)=\{b, c, d\}$, $V(C)=\{b, c, d, e\}$, $V(D)=\{a, b, c\}$  satisfy the first condition but not the second.} 
If $H$ is planar and with no $(\leq 2)$-separation, then the \defn{flap-number} of $H$ is defined to be 1. Otherwise, the \defn{flap-number} of $H$ is defined to be the maximum number of pairwise independent flaps in $H$.  Let $f(H)$ denote the flap-number of $H$.

The following is our main theorem. 

\begin{theorem}
\label{Main}
For every graph $H$ and every surface $\Sigma$ in which $H$ embeds, 
\begin{equation*}
    C(H,\Sigma,n) = \Theta( n^{f(H)} ).
\end{equation*}
\end{theorem}
    
It is immediate from the definitions that $f(H)=0$ if and only if $H$ is strongly non-planar. So \cref{StronglyNonPlanar} follows from the $f(H) \leq 1$ cases of \cref{Main}. 

As an aside, note that \cref{Main} can be restated as follows: for every graph $H$ and every surface $\Sigma$ in which $H$ embeds, 
\begin{equation*}
    \lim_{n\to\infty} \frac{ \log C(H,\Sigma,n)}{ \log n} = f(H) .
\end{equation*}

The above limit is sometimes referred to as the \defn{asymptotic logarithmic density} of $H$ in $\Sigma$.  A related result of \citet*{NesOss11} shows that for every infinite nowhere dense hereditary graph class $\mathcal{G}$ and for every fixed graph $H$, the maximum, taken over all $n$-vertex graphs $G\in\mathcal{G}$, of the number of induced subgraphs of $G$ isomorphic to $H$ is $\Omega( n^{\beta})$ and $O( n^{\beta+o(1)})$ for some integer $\beta\leq f(H)$. Our results (in the case that $\mathcal{G}$ is the class of graphs embeddable in a fixed surface) imply this upper bound  (since the number of induced copies of $H$ in $G$ is at most $C(H,G)$). Moreover, our bounds are often more precise since $f(H)$ can be significantly less than $\beta$. 

The lower bound in \cref{Main} is proved in  \cref{LowerBound}. 
\cref{Tools} introduces some tools from the literature that are used in the proof of the upper bound. \cref{StronglyNonPlanar} is proved in \cref{BoundedNumberCopies}. The upper bound in \cref{Main} is then proved in \cref{MainProof}. \cref{CompleteGraphs} presents more precise bounds on $C(H,\Sigma,n)$ when $H$ is a complete graph $K_s$. \cref{MinorClosedClasses} considers the maximum number of copies of a graph $H$ in an $n$-vertex graph in a given minor-closed class. \cref{Homomorphism} reinterprets our results in terms of homomorphism inequalities, and presents some open problems that arise from this viewpoint. 

Before continuing, to give the reader some more intuition about \cref{Main}, we now asymptotically determine $C(T,\Sigma,n)$ for a tree $T$. 

\begin{corollary}
\label{TreeCopies}
For every fixed tree $T$, let $\beta(T)$ be the size of a maximum stable set in the subforest $F$ of $T$ induced by the vertices with degree at most $2$. Then for every fixed surface $\Sigma$, 
\begin{equation*}
    C(T,\Sigma,n) = \Theta( n^{\,\beta(T)} ).
\end{equation*}
\end{corollary}

\begin{proof}
By \cref{Main}, it suffices to show that $\beta(T)=f(T)$. 

Let $I=\{v_1,\dots,v_{\beta(T)}\}$ be a maximum stable set in $F$. Let $x_i$ (and possibly $y_i$) be the neighbours of $v_i$. 
Let $A_i:=T[\{v_i,x_i,y_i\}]$ and $B_i:=T-v_i$. Then $(A_i,B_i)$ is a flap of $T$. Since $I$ is a stable set, for each $v_i\in I$ neither $x_i$ nor $y_i$ are in $I$, implying that $E(A_i^-)\cap E(A_j^-)=\emptyset$ for distinct $i,j\in [\beta(T)]$.  Moreover, $V(A_i) \setminus V(B_i)=\{v_i\}$, so  $(V(A_i) \setminus V(B_i)) \cap (V(A_j) \setminus V(B_j))=\emptyset$ for all distinct $i,j$. Hence $(A_1,B_1),\dots,(A_{\beta(T)},B_{\beta(T)})$ are pairwise independent flaps in $T$. Thus $\beta(T) \leq f(T)$. \cref{Main} then implies that 
$C(T,\Sigma,n) = \Omega( n^{\,\beta(T)} )$. This lower bound is particularly easy to see when $T$ is a tree. Let $G$ be the graph obtained from $T$ by replacing each vertex $v_i \in I$ by $\floor{\frac{n-|V(T)|}{\beta(T)}}$ vertices with the same neighbourhood as $v_i$, as illustrated in \cref{fig:TreeCopies}. Then $G$ is planar with at most $n$ vertices and at least $(\frac{n-|V(T)|}{\beta(T)})^{\beta(T)}$ copies of $T$. Thus $C(T,\Sigma,n) \geq C(T,\mathbb{S}_0,n) = \Omega( n^{\beta(T)} )$ for fixed $T$. 

For the converse, let $(A_1,B_1),\dots,(A_{f(T)},B_{f(T)})$ be pairwise independent flaps in $T$. Choose $(A_1,B_1),\dots,(A_{f(T)},B_{f(T)})$ to minimise $\sum_{i=1}^{f(T)} |V(A_i)|$. A simple case-analysis shows that $|V(A_i)\setminus V(B_i)|=1$, and if $v_i$ is the vertex in $V(A_i)\setminus V(B_i)$, then $N(v_i)= V(A_i)\cap V(B_i)$, implying $v_i$ has degree 1 or 2 in $T$. Moreover, $v_iv_j\not\in E(T)$ for distinct $i,j\in[f(T)]$ as otherwise $E(A_i^-)\cap E(A_j^-)\neq\emptyset$. Hence $\{v_1,\dots,v_{f(T)}\}$ is a stable set of vertices in $T$ all with degree at most 2. Hence $\beta(T)\geq f(T)$. 
\end{proof}

\begin{figure}[!ht]
\centering
\includegraphics{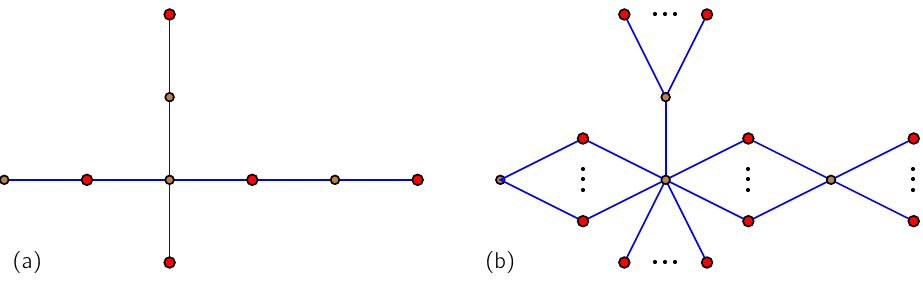}
\caption{(a) A tree $T$ with $\beta(T)=5$. (b) A planar graph with $\Omega(n^5)$ copies of $T$.
\label{fig:TreeCopies}}
\end{figure}

%\cref{Main} determines $f(H,\sigma,n)$ asymptotically. 

%This naturally leads to the question: for which graphs $H$ is $f(H,\Sigma,n)\in O(n)$? The following definition is the key to answering this question. A \defn{dangerous pair} in a graph $H$ consists of two distinct $(\leq 2)$-separations $(H_1,H_2)$ and $(J_1,J_2)$ such that both $H_1^+$ and $H^+_2$ are planar, and $H_1-V(H_2)$ and $J_1-V(J_2)$ are disjoint. The following is our second main contribution. It says that $f(H,\Sigma,n)\leq O(n)$ if and only if $H$ has no dangerous pair. 

%\begin{theorem} \label{LinearlyManyCopies} For every graph $H$ with no dangerous pair and for every surface $\Sigma$, there is a number $c$ such that every $n$-vertex graph embeddable in $\Sigma$ contains at most $cn$ copies of $H$; that is, $f(H,\Sigma,n)\leq cn$ for all $n$. Conversely, for every graph $H$ that contains a dangerous pair and for every surface $\Sigma$ in which $H$ embeds, there is a constant $c>0$ such that for all $n\geq ?????|V(H)|$, there is an $n$-vertex  graph that embeds in $\Sigma$ and contains at least $cn^2$ copies of $H$; that is, $f(H,\Sigma,n)\geq cn^2$. \end{theorem}

%\cref{LinearlyManyCopies} completely solves an open problem posed by \citet*{Eppstein93}. It is proved in \cref{LinearNumberCopies}. We emphasise that \cref{StronglyNonPlanar,LinearlyManyCopies} together show that $f(H,\Sigma,n)$ is $O(1)$ or $\Theta(n)$ or $\Omega(n^2)$ for every fixed graph $H$ (regardless of the surface $\Sigma$). 

%%%%%%%%%%%%%%%%%%%%%%%%%%%%%%%%%%
\section{Lower Bound}
\label{LowerBound}

Now we prove the lower bound in \cref{Main}. Let $H$ be an $h$-vertex graph with flap-number $k$. Let $\Sigma$ be a surface in which $H$ embeds. Our goal is to show that $C(H,\Sigma,n) = \Omega( n^k )$ for all $n\geq 4|V(H)|$.  If $k=0$, then there is nothing to prove.  If $H$ is planar and $3$-connected, then we may take $\lfloor n/h \rfloor$ disjoint copies of $H$.  Thus, we may assume that $H$ has at least one flap. Let $(A_1,B_1),\dots,(A_k,B_k)$ be pairwise independent flaps in $H$. If $(A_i,B_i)$ is a 1-separation, then let $v_i$ be the vertex in $A_i\cap B_i$. If $(A_i,B_i)$ is a 2-separation, then let $v_i$ and $w_i$ be the two vertices in $A_i\cap B_i$. 
Let $H'$ be obtained from $H$ as follows: if $(A_i,B_i)$ is a 2-separation, then delete $A_i-V(B_i)$ from $H$, and add the edge $v_iw_i$ (if it does not already exist). Note that $H'$ is a minor of $H$, since we may assume that whenever $(A_i,B_i)$ is a 2-separation, there is a $v_iw_i$-path in $A_i$ (otherwise $(A_i,B_i)$ can be replaced by a $(\leq 1)$-separation). 
Since $H$ embeds in $\Sigma$, so does $H'$. By assumption, $A_i^+$ is planar for each $i$. Fix an embedding of $A_i^+$ with $v_i$ and $w_i$ (if it exists) on the outerface (which exists since $v_iw_i$ is an edge of $A_i^+$ in the case of a 2-separation). Let $q:= \floor{ \frac{n}{|V(H)|}-1}$ and  $G$ be the graph obtained from an embedding of $H'$ in $\Sigma$ by adding $q$ disjoint copies of $A_i^+$ (if $(A_i, B_i)$ is a 0-separation), pasting 
$q$ copies of $A_i^+$ onto $v_i$ (if $(A_i,B_i)$ is a 1-separation), and pasting $q$ copies of $A_i^+$ onto  $v_iw_i$ (if $(A_i,B_i)$ is a 2-separation). These copies of $A_i^+$ can be embedded into a face of $H'$, as illustrated in \cref{QuadraticCopies}. 

Since $(V(A_i)\setminus V(B_i)) \cap (V(A_j)\setminus V(B_j)) = \emptyset$ for distinct $i,j\in[k]$, 
$$|V(G)| 
= |V(H)| + q \sum_i|V(A_i) \setminus V(B_i) |
\leq (q+1) |V(H)|
\leq n.$$
By construction, $G$ has at least $q^k  \geq ( \frac{n}{|V(H)|}-2 )^k $ copies of $H$. 
Hence $C(H,\Sigma,n) = \Omega(n^k)$.

%\begin{figure}[!ht]
%\centering
%\includegraphics{LinearCopies}
%\caption{Linearly many copies of $H$.
%\label{LinearCopies}}
%\end{figure}

%First we prove the lower bound in \cref{3ConnectedPlanar} (due to \citet*{Eppstein93}). Say $H$ is a planar graph that has a $(\leq 2)$-separation $(H_1,H_2)$ such that $H_1$ and $H_2$ is planar. Start from an embedding of $H$ in the sphere with $V(H_1\cap H_2)$ on the outerface. Then, as illustrated in \cref{QuadraticCopies}, $\floor{\frac{n}{h}}$ copies of $H_1$ and $\floor{\frac{n}{h}}$ copies of $H_2$ can be pasted on $H_1\cap H_2$ to form a graph $G$ embedded in the sphere with at least $\floor{\frac{n}{h}}^2$ copies of $H$. Note that $|V(G)|=\floor{\frac{n}{h}}(h_1-2) + \floor{\frac{n}{h}}(h_2-2)+ 2 = \floor{\frac{n}{h}}(h_1+h_2-4) + 2 = \floor{\frac{n}{h}}(h-2) + 2 \leq \frac{h-2}{h}\,n + 2 \leq n$. Since $G$ can be embedded in any surface, it follows that $f(H,\Sigma,n)\geq \frac{n^2}{4h^2}$. 

\begin{figure}[!ht]
\centering
\includegraphics{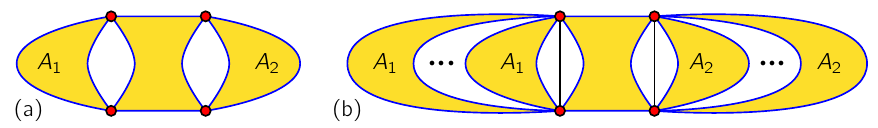}
\caption{(a) A graph $H$ with flap-number 2. (b) A graph with $\Omega(n^2)$ copies of $H$.
\label{QuadraticCopies}}
\end{figure}

\section{Tools}
\label{Tools}

%Let $H$ be a graph, $\Sigma$ be a surface, and $n \in \N$.  We let $f(H, \Sigma, n)$ be the maximum number of copies of $H$ in $G$, where $G$ ranges over all $n$-vertex graphs embedded on $\Sigma$. We let $f(H, t, n)$ be the maximum number of copies of $H$ in $G$, where $G$ ranges over $K_{3,t}$-minor-free graphs with $n$ vertices.  Note that every graph with Euler genus at most $g$ does not contain a $K_{3,2g+3}$-minor.

In Sections~\ref{Tools}--\ref{MainProof} of this paper we work in the following setting. For graphs $G$ and $H$, an  \defn{image} of $H$ in $G$ is an injection $\phi: V(H) \to V(G)$ such that $\phi(u)\phi(v) \in E(G)$ for all $uv \in E(H)$.  Let  $I(H,G)$ be the number of images of $H$ in $G$, and let $I(H,\Sigma,n)$ be the maximum of $I(H,G)$ taken over all $n$-vertex graphs $G$ that embed in $\Sigma$. If $H$ is fixed then $C(H,G)$ and $I(H,G)$ differ by a constant factor. In particular, if $|V(H)|=h$ then 
\begin{align*}
C(H,G)& \leq I(H,G) \leq h!\, C(H,G).\\
C(H,\Sigma,n)& \leq I(H,\Sigma,n) \leq h!\, C(H,\Sigma,n).
\end{align*}
So to prove our main theorems, it suffices to work with images rather than copies. 

To prove the upper bound in \cref{Main} we  need several tools from the literature. The first two were proved by \citet*{Eppstein93}. To state the first result we need the following definition.  
%First, to be formal, a \defn{copy} of $H$ in $G$ is an injection $\phi: V(H) \to V(G)$ such that $\phi(u)\phi(v) \in E(G)$ for all $uv \in E(H)$.  
A collection $\mathcal{H}$ of images of $H$ in $G$ is \defn{coherent} if for all images $\phi_1, \phi_2 \in \mathcal{H}$ and for all distinct vertices $x,y\in V(H)$, we have $\phi_1(x) \neq \phi_2(y)$.

\begin{lemma}[\citep{Eppstein93}] \label{coherence}
Let $H$ be a graph with $h$ vertices and $G$ be a graph.   Every collection of at least $
c_{\ref{coherence}}(h,t):=h!^2 t^h$ images of $H$ in $G$ contains a coherent subcollection of size at least $t$. 
\end{lemma}

% {\david In \cref{coherence}, I suggest we write \\
% Let $H$ be a graph with $h$ vertices and $G$ be a graph.  Every collection of at least $
% c_{\ref{coherence}}(h,t):=h!^2 t^h$ images of $H$ in $G$ contains a coherent subcollection of size at least $t$.  \\
% or\\
% There is a function $c_{\ref{coherence}}(h,t)$ such that for every graph $H$ with $h$ vertices and for every graph $G$, every collection of at least $c_{\ref{coherence}}(h,t)$ images of $H$ in  $G$ contains a coherent subcollection of size at least $t$.  \\
% Then use $c_{\ref{coherence}}(h,t)$ later. For example, we would define
% $c_{\ref{BoundedCopies}}(h,g):=
% c_{\ref{coherence}}(h, c_{\ref{sunflower}}(h,2g+3))$. 
% }

% {\tony I like this idea.  I started implementing it.}

% We need this for the $K_{3,g}$-minor case.  

% \begin{theorem} \label{Eppstein}
% Let $H$ be a graph with no $(\leq 2)$-separation.  For every graph $G$ with Euler genus at most $g$, and for every vertex $x$ of degree $d$ in $G$, there are less than $h!(2g+3)^{h^2}d^{3h^2}$ images of $H$ in $G$ all containing $x$.
% \end{theorem}

\begin{theorem}[\citep{Eppstein93}] \label{EppsteinCor}
There exists a function $c_{\ref{EppsteinCor}}(h,g)$ such that for every planar graph $H$ with $h$ vertices and no $(\leq 2)$-separation, and every surface $\Sigma$ of Euler genus $g$, 
\[
I(H,\Sigma, n) \leq c_{\ref{EppsteinCor}}(h,g) n.
\]
\end{theorem}

The next key tool is the following result by \citet*{Miller-JCTB87} and \citet*{Archdeacon-JGT86}.

\begin{theorem}[Additivity of Euler genus~\citep{Miller-JCTB87,Archdeacon-JGT86}] 
\label{Additivity}
For all graphs $G_1$ and $G_2$, if $|V(G_1)\cap V(G_2)|\leq 2$ then the Euler genus of $G_1\cup G_2$ is at least the Euler genus of $G_1$ plus the Euler genus of $G_2$.  
\end{theorem}

We also use the following result of \citet*{ER60}; see~\citep{ALWZ} for a recent quantitative improvement. A \defn{$t$-sunflower} is a collection $\mathcal{S}$ of $t$ sets for which there exists a set $R$ such that $X\cap Y=R$ for all distinct $X,Y\in\mathcal{S}$. The set $R$ is called the \defn{kernel} of $\mathcal{S}$. 

\begin{lemma}[Sunflower Lemma~\citep{ER60}] \label{sunflower}
There exists a function $c_{\ref{sunflower}}(h,t)$ such that every collection of $c_{\ref{sunflower}}(h,t)$ many $h$-subsets of a set contains a $t$-sunflower.  
\end{lemma}

Finally, we mention some well-known corollaries of Euler's Formula that we use implicitly. Every graph with $n\geq 3$ vertices and Euler genus $g$ has at most $3(n+g-2)$ edges. Moreover, for bipartite graphs the above bound is $2(n+g-2)$. For example, this implies that the complete bipartite graph $K_{3,2g+3}$ has Euler genus greater than $g$.

\section{Strongly Non-Planar Graphs}
\label{BoundedNumberCopies}

We begin by proving a quantitative version of the upper bound in \cref{StronglyNonPlanar}. In fact, we will prove that \cref{StronglyNonPlanar} holds more generally for what we call `partially subdivided graphs'.  A \defn{partially subdivided graph} is a pair $(H, \mathcal P)$, where $H$ is a graph and $\mathcal P$ is a collection of internally disjoint paths in $H$ such that the two ends of each path in $\mathcal P$ are not adjacent in $H$, and every internal vertex of each path in $\mathcal P$ has degree $2$ in $H$.  Let $H - \mathcal P$  be the graph obtained from $H$ by deleting every internal vertex of each path in $\mathcal P$.  Finally, let $H / \mathcal P$ be the minor of $H$ obtained by contracting all but one edge from each path in $\mathcal P$. 

\begin{theorem}
\label{BoundedCopies}
Let $c_{\ref{BoundedCopies}}(h,g):=
c_{\ref{coherence}}(h, c_{\ref{sunflower}}(h,2g+3))$. 
Then for every partially subdivided graph $(H, \mathcal P)$ such that $H / \mathcal P$ is strongly non-planar and $|V(H)|=h$, for every surface $\Sigma$ with Euler genus $g$, and for every graph $G$ embedded in $\Sigma$,  there are at most
$c_{\ref{BoundedCopies}}(h,g)$ images of $H - \mathcal P$ in $G$ that extend to an image of $H$ in $G$.
\end{theorem}

\begin{proof}
Assume for the sake of contradiction, that there is a collection $\mathcal{H}$ of more than $c_{\ref{BoundedCopies}}(h,g)$ images of $H$ in $G$, such that all restrictions of these images to $H - \mathcal P$ are distinct.  By \cref{coherence}, $\mathcal H$ contains a coherent subfamily $\mathcal H_0$ of size at least $c_{\ref{sunflower}}(h,2g+3)$.  Let $\mathcal V$ be the collection of vertex sets of the images of $H$ in $\mathcal H_0$.  By coherence, $|\mathcal V|=|\mathcal H_0| \geq c_{\ref{sunflower}}(h,2g+3)$.

 By the Sunflower Lemma, $\mathcal V$ contains a $(2g+3)$-sunflower $\mathcal{F}$.  
 (Abusing notations slightly, we equate sets in $\mathcal F$ with the corresponding images of $H$ in $\mathcal H_0$.) 
 Let $Z$ be the kernel of $\mathcal{F}$. Thus $F \cap F'=Z$ for all $F, F' \in \mathcal{F}$.   Let $I$ be the set of internal vertices of all $P \in \mathcal P$.  Since the restrictions of the images of $H$ in $\mathcal F$ to $H - \mathcal P$ are all distinct, 
 for each $F \in \mathcal F$ there exists a vertex $w_F \in F \setminus Z$ such that $w_F$ is the image of a vertex in $V(H) \setminus I$.  By coherence, we may assume that every $w_F$ is the image of the same vertex of $V(H) \setminus I$.  For each $F \in \mathcal F$, let $C_F$ be the component of $F - Z$ that contains $w_F$, and let $N_F$ be the vertices of $Z$ with at least one neighbour in $V(C_F)$.  By coherence, $N_F$ is the same for all $F \in \mathcal F$.  Therefore, we obtain a $K_{|N_F|, |\mathcal F|}$ minor in $G$ by contracting each $C_F$ to a vertex.  Since $|\mathcal F| = 2g+3$ and $K_{3, 2g+3}$ does not embed in $\Sigma$, we must have $|N_F| \leq 2$. 
 
 For each $F \in \mathcal F$, consider the pair of subgraphs $(F_1, F_2):=(F[Z], F[(V(F) \setminus Z) \cup N_F])$ of $F$.  Since $|N_F| \leq 2$, either $(F_1, F_2)$ is a $(\leq 2)$-separation of $F$ or $F_2=F$.  For each $P \in \mathcal{P}$ choose $e_P \in E(P)$ and let $E_{\mathcal{P}}:=\bigcup_{P \in \mathcal{P}} E(P) \setminus \{e_P\}$.  
 For $i \in [2]$, let $\Gamma_i:=F_i / (E(F_i) \cap E_{\mathcal{P}})$.   Since $w_F \in V(F_2) \setminus V(F_1)$, it follows that either $(\Gamma_1, \Gamma_2)$ is a $(\leq 2)$-separation of $F / \mathcal P$ or $\Gamma_2 \cong F / \mathcal P $.  Since $F / \mathcal P \cong H / \mathcal P$ and $H / \mathcal P$ is strongly non-planar, either $(\Gamma_2)^+$ is non-planar or $\Gamma_2 = F / \mathcal P $.  In the first case, $F_2^+$ is also non-planar since $\Gamma_2^+$ is a minor of $F_2^+$.  In the second case,  $F_2=F$ is also  non-planar.  Let $F_2':=F_2^+$ if the first case holds, and $F_2'=F$ if the second case holds.  If $N_F:=\{x,y\}$ and $xy \notin E(F)$, let $\Sigma'$ be obtained from $\Sigma$ by adding a handle and using the handle to draw the edge $xy$, and let $G':=G \cup \{xy\}$.  Otherwise, let $\Sigma':=\Sigma$ and $G':=G$.
 We conclude by noting that the family of subgraphs $\{F_2' \mid F \in \mathcal F\}$ of $G'$ contradicts the Additivity of Euler genus (\cref{Additivity}), since they each have Euler genus at least $1$, they pairwise intersect only in $N_F$, are drawn on a surface $\Sigma'$ of Euler genus at most $g+2$, and $|\mathcal F| = 2g+3 >g+2$.  
 \end{proof}

Note that if $H$ is a strongly non-planar graph, then we recover \cref{StronglyNonPlanar} by applying \cref{BoundedCopies} to the partially subdivided graph $(H, \emptyset)$. We need the stronger statement in \cref{BoundedCopies} for the proof of \cref{main} to come. 

\section{Proof of Main Theorem}
\label{MainProof}
The proof of our main theorem uses a variant of the SPQR tree, which we now introduce.

\subsection{SPQRK Trees}
\label{SPQRK}
    
The \defn{SPQR tree} of a $2$-connected graph $G$ is a tree that displays all the $2$-separations of $G$.  Since we need to consider graphs that are not necessarily $2$-connected, we use a variant of the SPQR tree that we call the \defn{SPQRK tree}.  
    
Let $G$ be a connected graph.  The \defn{SPQRK tree $T_G$} of $G$ is a tree, where each node $a\in V(T_G)$ is associated with a multigraph $H_a$ which is a minor of $G$. Each vertex $x \in V(H_a)$ is a vertex of $G$, that is, $V(H_a) \subseteq V(G)$. 
    Each edge $e \in E(H_a)$ is classified either as a \defn{real} or \defn{virtual} edge. By the construction of an SPQRK tree each edge $e\in E(G)$ appears in exactly one minor $H_a$ as a real edge, and each real edge $e\in E(H_a)$ is an edge of $G$. 
    The SPQRK tree $T_G$ is defined recursively as follows.
    \begin{enumerate}
        \item If $G$ is $3$-connected, then $T_G$ consists of a single \defn{$R$-node} $a$ with $H_a := G$. 
        All edges of $H_a$ are real in this case.
        
        \item If $G$ is a cycle, then $T_G$ consists of a single \defn{$S$-node} $a$ with $H_a := G$. 
        Again, all edges of $H_a$ are real in this case. 
        
        \item If $G$ is isomorphic to $K_1$ or $K_2$, then $T_G$ consists of a single \defn{$K$-node} $a$ with $H_a := G$. Again, all edges of $H_a$ are real in this case. 
        
        \item If $G$ is $2$-connected and has a cutset $\{x,y\}$ such that the vertices $x$ and $y$ have degree at least $3$, we construct $T_G$ inductively as follows. Let $C_1,\dots, C_r$ ($r\geq 2$) be the connected components of $G - \{x,y\}$. First add a \defn{$P$-node} $a$ to $T_G$, for which $H_a$ is the graph with $V(H_a):=\{x,y\}$ consisting of $r$ parallel virtual edges and one additional real edge if $xy$ is an edge of $G$.
        
        Next let $G_i$ be the graph $G[V(C_i)\cup \{x,y\}]$ with the additional edge $xy$ if it is not already there. Since we include the edge $xy$, each $G_i$ is $2$-connected and we can construct the corresponding SPQRK tree $T_{G_i}$ by induction. Let $a_i$ be the (unique) node in $T_{G_i}$ for which $xy$ is a real edge in $H_{a_i}$. 
        In order to construct $T_G$, we make $xy$ a virtual edge in the node $a_i$,  and connect $a_i$ to $a$ in $T_G$. 
        
        \item If $G$ has a cut-vertex $x$ and $C_1, \dots, C_s$ ($s\geq 2$) are the connected components of $G - x$, then construct $T_G$ inductively as follows.  First, add a  \defn{$Q$-node} $a$ to $T_G$, for which $H_a$ is the graph consisting of the single vertex $x$.  For each $i \in [s]$, let $G_i := G[V(C_i)\cup \{x\}]$.  Since $G_i$ is connected, we can construct the corresponding SPQRK tree $T_{G_i}$ by induction.  If there is a unique node $b_i \in V(T_{G_i})$ such that $x \in V(H_{b_i})$, then make $a$ adjacent to $b_i$ in $T_{G}$. If $x$ is in at least two nodes of $V(T_{G_i})$, then $x \in V(C) \cap V(D)$ for some $(\leq 2)$-separation $(C,D)$ of $G_i$.  Since $G_i - x$ is connected, there must be a $P$-node $b_i$ in $T_{G_i}$ such that $x \in V(H_{b_i})$. Note that $b_i$ is not necessarily unique.  Choose one such $b_i$ and make $a$ adjacent to $b_i$ in $T_G$.
    \end{enumerate} 
    
As a side remark, note that the SPQRK tree $T_G$ of $G$ is in fact not unique---there is some freedom in choosing $b_i$ in the last point in the definition above---however, for our purposes we do not need uniqueness, we only need that $T_G$ displays all the $(\leq 2)$-separations of $G$.

The next lemma is the crux of the proof.  Let $J$ and $G$ be graphs and $X$ and $Y$ be cliques in $J$ and $G$ respectively, with $|X|=|Y|$.  Let $\phi:V(J) \to V(G)$ be an image of $J$ in $G$.  We say that \defn{$\phi$ fixes $X$ at $Y$} if $\phi(X)=Y$.  Let $(J', \mathcal P)$ be a partially subdivided graph such that $J=J'/ \mathcal P$.  We call $uv \in E(J)$ a \defn{fake} edge if $u$ and $v$ are the set of ends of some $P \in \mathcal P$.  Otherwise, $uv$ is a \defn{true edge}. 
% Let $K_4^-$ denote $K_4$ minus an edge.  
 
\begin{lemma} \label{lem:rootedclique}
Let  $c_{\ref{lem:rootedclique}}(j,g) :=12(g+1)c_{\ref{coherence}}(j, c_{\ref{sunflower}}(j, 2g+3))$. 
Let $\Sigma$ be a surface of Euler genus $g$.  Let $(J', \mathcal P)$ be a connected, partially subdivided planar graph with $|V(J')|=j$ and let $J:=J'/ \mathcal P$.

Let $X$ be a clique in $J$ such that:
\begin{enumerate}
   \item \label{star}there do not exist independent flaps $(A,B)$ and $(C,D)$ of $J$ with $X \subseteq V(B \cap D)$, 
  \item \label{trueX} $|X| \in \{1,2\}$, and if $|X|=2$, then $X$ is a true edge,
%   \item if $|X|=1$ and $J\cong K_2$, then $e$ is a true edge, where $e$ is the unique edge of $J$,
  \item \label{smallp3} if $|X|=1$ and $J \cong P_3$, then $e$ is a true edge, where $e$ is the unique edge of $J$ not incident to $X$, 
  \item \label{smallc3} if $|X|=1$ and $J \cong C_3$, then $e$ is a true edge, where $e$ is the unique edge of $J$ not incident to $X$,
 \item \label{smallc4} if $|X|=2$ and $J \cong C_4$, then $e$ is a true edge, where $e$ is the unique edge of $J$ not incident to a vertex of $X$,
%  \item \label{k4minus} if $X=1$ and $J \cong K_4^-$, then there exists a true edge $uv$ such that $\deg_J(w)=3$, where $w$ is the unique vertex in $V(J) \setminus (X \cup \{u,v\})$,
 \item \label{item:true3conn} if $J$ is $3$-connected, then all edges of $J$ with neither end in $X$ are true,
  \item \label{p3trueflap} if $(A,B)$ is a flap of $J$, with $X \subseteq V(B), |V(A \cap B)|=1$, and $A \cong P_3$, then $e$ is a true edge, where $e$ is the unique edge of $A$ not incident to $V(A \cap B)$, 
%   \item \label{truedeg2} if $x$ is a degree-$2$ vertex of $J$, then at least one of the edges incident to $x$ is a true edge, 
 \item \label{c3trueflap} if $(A,B)$ is a flap of $J$, with $X \subseteq V(B), |V(A \cap B)|=1$, and $A \cong C_3$, then $e$ is a true edge, where $e$ is the unique edge of $A$ not incident to $V(A \cap B)$, 
 \item \label{c3trueflap2} if $(A,B)$ is a flap of $J$, with $X \subseteq V(B), |V(A \cap B)|=2$, and $A^+ \cong C_3$, then at least one $e$ or $f$ is a true edge, where $e$ and $f$ are the two edges of $A$ with an end not on $V(A \cap B)$.  
 \item \label{c4trueflap} if $(A,B)$ is a flap of $J$, with $X \subseteq V(B), |V(A \cap B)|=2$, and $A^+ \cong C_4$, then $e$ is a true edge, where $e$ is the unique edge of $A$ not incident to $V(A \cap B)$, 
 \item \label{true3flap} if $(A,B)$ is a flap of $J$ such that $A^+$ is $3$-connected, then all edges of $A$ with neither end in $V(A \cap B)$ are true.  
 \end{enumerate}
 Then for every $n$-vertex graph $G$ embeddable in $\Sigma$ and every clique $Y$ in $G$ with $|Y|=|X|$, there are at most $c_{\ref{lem:rootedclique}}(j,g) n$ images of $J' - \mathcal P$ in $G$ with $X$ fixed at $Y$ that extend to an image of $J'$ in $G$.
 \end{lemma}

\begin{proof}[Proof of \cref{lem:rootedclique}]
Let $G$ be an $n$-vertex graph embedded in a surface $\Sigma$ of Euler genus $g$ and $Y$ be a clique in $G$ with $|Y|=|X|$. 
We begin by proving the lemma when $J$ is small. The lemma clearly holds if $|V(J)|=1$ or $|V(J)|=2=|X|$.  If $|V(J)|=2$ and $|X|$=1, let $y$ be the vertex of $J$ not in $X$.  Since there are at most $n-1$ vertices of $G$ to send $y$ to, we are done.  Similarly, we are done if $J \in \{P_3, C_3\}$ and $|X|=2$.  If $J \in \{P_3, C_3\}$ and $|X|=1$ let $e$ be the unique edge of $J$ not incident to $X$.  Note that $e$ exists in the case that $J \cong P_3$, since the vertex in $X$ cannot be the middle vertex of $J$ by \eqref{star}.  By (\ref{smallp3}) and (\ref{smallc3}), $e$ is a true edge, so there are at most $|E(G)| \leq 3(g+1)n$ edges of $G$ to send $e$ to.  Each edge gives at most two images of $J' - \mathcal P$ with $X$ fixed at $Y$ in $G$, so there are at most $6(g+1)n$ such images.  Suppose that $|V(J)|=4$.  Note that $J \not\cong P_4$, by \eqref{star}.  If $J \cong C_4$, then $|X| =2$ by \eqref{star}. Let $e$ be the unique edge of $J$ not incident to a vertex of $X$.  By (\ref{smallc4}), $e$ is a true edge, so again there are at most $3(g+1)n$ edges of $G$ to send $e$ to.  Each edge gives at most four images of $J' - \mathcal P$ with $X$ fixed at $Y$ in $G$, so there are at most $12(g+1)n$ such images.  

In summary, by the above discussion we may assume that $|V(J)|\geq 4$, and $J \not\cong P_4, C_4$ in case $|V(J)|=4$. 

Let $T_J$ be the SPQRK tree of $J$.  Suppose that $V(T_J)=\{a\}$.  If $a$ is a $K$-node, then we are done since $|V(J)| \leq 2$.  If $a$ is an $S$-node, then by \eqref{star}, $J\cong C_3$ or $J\cong C_4$, so we are done.  By the preceding remarks, we may assume that $J$ is $3$-connected, or $|V(J)| \geq 4$ and $|V(T_J)| \geq 2$.  
A clique $X'$ of $J$ is a \defn{true clique} if $|X'|=1$, or $|X'|=2$ and the edge of $X'$ is a true edge.  If $J$ is $3$-connected, we have the following easy claim.

\begin{claim} \label{menger}
If $J$ is $3$-connected, then there exists a true clique $X'$ in $J$ such that 
% $V(J) \setminus (X \cup X') \neq \emptyset$, and 
for all $w \in V(J) \setminus (X \cup X')$, there are three internally disjoint paths in $J$ from $w$ to $X \cup X'$, whose ends in $X \cup X'$ are distinct. 
\end{claim}

\begin{proof}
Since $J$ is $3$-connected, there is an edge $X'$ of $J$ with neither end in $X$; otherwise, $X$ is a cutset of $J$ with $|X| \leq 2$. By (\ref{item:true3conn}), $X'$ is a true edge. Since $J$ is $3$-connected, we are done by Menger's theorem.  
% Suppose $J \cong K_4^-$.  If $|X|=1$, then $x:=\{X\}$
% must be a degree-$2$ vertex of $J$ by $(\star)$.  Thus, we may take $X'$ to be the true edge $uv$ guaranteed to exist by (\ref{k4minus}).  If $|X|=2$, then one vertex of $X$ is a degree-$2$ vertex of $J$ by $(\star)$.  We let $X'=\{x\}$, where $x$ is the other degree-$2$ vertex of $J$.  
\end{proof}

We now suppose that $|V(J)| \geq 4$ and $|V(T_J)| \geq 2$, and we prove that \cref{menger} also holds in this case. Let $W$ be the set of $K$-, $S$-, and $R$-nodes of $V(T_J)$.  Note that for each $a \in W$, $H_a$ is a subgraph of $J$.  If $U$ is a non-empty proper subset of $W$, we define $H_U := \bigcup_{a \in U} H_a$, $\bd(U) := V(H_U \cap H_{W \setminus U})$, $\lambda(U) :=|\bd(H_U)|$, and $\sep(U):=(H_U, H_{W \setminus U})$.  
% The next three claims follow from  \eqref{star}.

\begin{claim}
 $T_J$ is a path and there is a leaf $\ell$ of $T_J$ such that $X \subseteq V(H_\ell)$ and $X \setminus \bd(\{\ell\}) \neq \emptyset$. 
\end{claim}

\begin{proof}
Since $|V(T_J)| \geq 2$, $T_J$ has at least two leaves.  For each leaf $a$ of $T_J$, there is a flap $(C^a, D^a)$ such that $V(C^a)=V(H_a)$.  Since $(C^a, D^a)$ and $(D^a, C^a)$ are independent flaps, exactly one of $(V(C^a) \setminus V(D^a)) \cap X$ or $(V(D^a) \setminus V(C^a)) \cap X$ is non-empty by \eqref{star}.  Thus, $T_J$ has exactly two leaves and there is a leaf $\ell$ of $T_J$ such that $X \subseteq V(H_\ell)$ and $X \setminus \bd(\{\ell\}) \neq \emptyset$.
\end{proof}

\begin{claim} \label{claim:lambda3}
Let $r$ be the other leaf of $T_J$. Then for all non-empty $U \subseteq W \setminus \{\ell, r\}$ such that $U$ is not a single $K$-node, $\lambda(U) \geq 3$.  
\end{claim}

\begin{proof}
Towards a contradiction, suppose that $\lambda(U) \leq 2$ for some non-empty $U \subseteq W \setminus \{\ell, r\}$ which is not a single $K$-node.  Let $(C^r, D^r)$ be the flap of $J$ such that $V(C^r)=V(H_r)$.  Since $X \subseteq V(H_\ell)$, $(C^r, D^r)$ and $\sep(U)$ are independent flaps of $J$ which contradict \eqref{star}.
\end{proof}

% The next claim also follows from \eqref{star}.  For completeness, we include the proof.  

\begin{claim} \label{claim:degree2}
Let  $S:=\{s \in V(J) \setminus X \mid \deg_J(s) \leq 2\}$.  Then $|S| \leq 2$, $S \subseteq V(H_r)$, and
         if $|S|=2$, then the two vertices in $S$ are adjacent in $J$.
\end{claim}

\begin{proof}
Since $|V(J)| \geq 4$, for each $s \in S$, $(\delta(s), J - s)$ is a flap with $X \subseteq V(J - s)$, where $\delta(s)$ is the subgraph of $J$ induced by the edges incident to $s$.  
Thus, by \eqref{star}, $S$ is a clique in $J$, and therefore $|S| \leq 3$.  Moreover, $|S|=3$ is impossible, since $|V(J)| \geq 4$ and $J$ is connected.  Thus, $|S| \leq 2$. 
Since $(A,B)=\sep(\{r\})$ is a flap with $X \subseteq V(B)$, \eqref{star} also implies $S \subseteq V(H_r)$.
\end{proof}

\begin{claim} \label{claim:cut-vertices}
$J$ has at most two cut-vertices.  Moreover, if $J$ has two cut-vertices, then they are the vertex set of some $K$-node of $T_J$.
\end{claim}

\begin{proof}
 Let $c$ and $d$ be distinct cut-vertices of $J$. Let $W_{cd} \subseteq W$ be the set of $K$-, $S$-, and $R$-nodes of $V(T_J)$ strictly between the $Q$-nodes corresponding to $c$ and $d$ in $T_J$.  Note that $\sep(W_{cd})$ is a $2$-separation of $J$, unless $W_{cd}$ is just a single $K$-node. Moreover, if $W_{cd}$ is not a single $K$-node, then $\sep(\{r\})$ and $\sep(W_{cd})$ would contradict \eqref{star}. It follows that $J$ has at most two cut-vertices, and that $\{c,d\}$ is the vertex set of some $K$-node of $T_J$.
\end{proof}

\begin{claim} \label{claim:3paths}
There exists a true clique $X'$ in $J$ such that 
% $V(J) \setminus (X \cup X') \neq \emptyset$, and 
for all $w \in V(J) \setminus (X \cup X')$, there are three internally disjoint paths in $J$ from $w$ to $X \cup X'$, whose ends in $X \cup X'$ are distinct. 
\end{claim}

\begin{proof}
If $r$ is an $R$-node, then all edges of $H_{r}^-$ are true by (\ref{true3flap}). 
In this case we let $X'$ be any edge of $H_{r}^-$ such that $X' \cap \bd(\{r\})=\emptyset$. If $r$ is an $S$-node, then by \eqref{star}, either $H_r \cong C_3$, or $H_r \cong C_4$ and $|\bd(\{r\})|=2$. In either case, by (\ref{c3trueflap}), (\ref{c3trueflap2}), and (\ref{c4trueflap}), we can choose $X'$ to be a true edge such that $V(H_r) \setminus (X' \cup \bd(\{r\})) = \emptyset$. Lastly, suppose that $r$ is a $K$-node. Then $H_r \cong K_2$, say $H_r$ consists of the edge $uv$ with $v \in \bd(\{r\})$.  If $v \notin S$, we let $X':=\{u\}$.  If $v \in S$, we let $X':=\{u,v\}$; note that $uv$ is a true edge by (\ref{p3trueflap}) in this case.  

% By Claim~\ref{claim:degree2}, there exists an edge $e=uv \in E(H_r)$ such that $S \subseteq \{u,v\}$. Among all such edges, choose $e=uv$ so that $|\{u,v\} \cap \bd(H_r)|$ is minimum.  Suppose the claim fails for this choice of $uv$.  

Suppose the claim is false for the above choice of $X'$ for some vertex $w \in V(J) \setminus (X \cup X')$, and let $X^+:=X \cup X'$. Note that $|X^+| \geq 3$ by our choice of $X'$, since if $|X|=1$ then $J$ is $2$-connected by (\ref{star}). 
(Indeed, if $J$ is not $2$-connected then $J$ has a flap $(A,B)$ that is a $1$-separation with $X \subseteq B$, but then $(B, A \cup X)$ is also a flap, contradicting (\ref{star}).)  
Thus, by Menger's theorem, there is a $(\leq 2)$-separation $(J_1,J_2)$ of $J$ with $w \in V(J_1) \setminus V(J_2)$ and $X^+:=X \cup X' \subseteq V(J_2)$. 

Let $a$ be an $S$-node of $T_J$. Observe that every $2$-separation of the cycle $H_a$ lifts to a $2$-separation of $J$. We say that a $2$-separation of $J$ is \defn{rooted at $a$} if it is a lift of a $2$-separation of $H_a$.  
Since the SPQRK tree $T_J$ of $J$ `displays' all the $(\leq 2)$-separations of $J$, every $(\leq 2)$-separation $(A,B)$ of $J$
\begin{itemize}
    \item is equal to $\sep(U)$ for some $U \subseteq W$, or 
    \item is rooted at some $S$-node $a$ of $T_J$, or 
    \item is obtained from a $1$-separation $(A',B')$ by adding an isolated vertex to $A'$ or $B'$ (which is thus in $A \cap B$).
 \end{itemize}
 
  Suppose $(J_1, J_2)=\sep(U)$ for some $U \subseteq W$. Since $X^+ \subseteq V(J_2)$, $X \setminus \bd(\{\ell\}) \neq \emptyset$, and $X' \setminus \bd(\{r\}) \neq \emptyset$, we have $U \subseteq W \setminus \{\ell, r\}$.  This is a contradiction since $\lambda(U) \geq 3$ by \cref{claim:lambda3}.  Similarly, $(J_1, J_2)$ cannot be rooted at an $S$-node,  unless $\deg_J(x)=2$ for some $x \in V(J_1) \setminus V(J_2)$.  However, by \cref{claim:degree2} and our choice of $X'$,  $\deg_J(x) \geq 3$, for all $x \in V(J) \setminus X^+$, so this is also impossible.  

Finally, suppose the third possibility holds for some $1$-separation $(A',B')$ of $J$.
Observe that every $1$-separation $(A,B)$ of $J$ has $X' \subseteq V(A)$ and $X \subseteq V(B)$; or $X \subseteq V(A)$ and $X' \subseteq V(B)$.  By swapping the order of  $(A',B')$ we may assume that $X' \subseteq V(A')$ and $X \subseteq V(B')$.  
Moreover, $(A',B' \cup \{a\}) \in \{(J_1, J_2), (J_2, J_1)\}$ for some $a \in V(A') \setminus V(B')$; or $(A' \cup \{b\}, B') \in \{(J_1, J_2), (J_2, J_1)\}$ for some $b \in V(B') \setminus V(A')$.  If $(A' \cup \{b\}, B')=(J_1, J_2)$, then $(A',B')$ is a $1$-separation of $J$ such that $X \cup X' \subseteq V(B')$. However, no such separation exists (by the proof that $(J_1, J_2) \neq \sep(U)$ for all $U \subseteq W$).  Similarly, $(A',B' \cup \{a\})=(J_2, J_1)$ is impossible.     
    If $(A' \cup \{b\}, B')=(J_2, J_1)$, then $(A',B')$ and $(B', A' \cup \{b\})$ contradict \eqref{star}. The remaining case is $(A', B' \cup \{a\})=(J_1, J_2)$.  Let $c$ be the unique vertex in $V(A') \cap V(B')$. Recall that by the choice of $X'$, if $r$ is an $R$-node or an $S$-node, then $|X'|=2$ and $c \notin X'$.  However, this contradicts $X' \subseteq V(J_2)$.  Thus, $r$ is a $K$-node.  Note that $(A',B')=\sep(\{r\})$ is impossible, because $V(A') \setminus (V(B' \cup \{a\}))$ would be empty, and hence $(A',B' \cup \{a\})$ is not a $2$-separation.  Thus $c \notin V(H_r)$.  Let $v$ be the cut-vertex of $J$ in $V(H_r)$. By \cref{claim:cut-vertices}, $c$ and $v$ are adjacent and $v \in S=\{s \in V(J) \setminus X \mid \deg_J(s) \leq 2\}$.      
    In this case, by our choice of $X'$, we have $|X'|=2$ and $c \notin X'$, so we again have a contradiction.  
\end{proof}

Let $X'$ be the true clique of $J$ given by \cref{menger} or by \cref{claim:3paths}, depending whether $J$ is $3$-connected, or $|V(J)| \geq 4$ and $|V(T_J)| \geq 2$.    
Suppose $|X'|=1$. For each $y \in V(G)$, let $c_y$ be the number of images of $J' - \mathcal P$ in $G$, with $X$ fixed at $Y$ and $X'$ fixed at $y$, that extend to an image of $J'$ in $G$. 
Suppose $|X'|=2$. For each $f \in E(G)$, let $c_f$ be the number of images of $J' - \mathcal P$ in $G$, with $X$ fixed at $Y$ and $X'$ fixed at $f$, that extend to an image of $J'$ in $G$.  

We claim that if $|X'|=1$ then $c_y \leq c_{\ref{coherence}}(j, c_{\ref{sunflower}}(j, 2g+3))$ for all $y \in V(G)$, if $|X'|=2$ then $c_f \leq c_{\ref{coherence}}(j, c_{\ref{sunflower}}(j, 2g+3))$ for all $f \in E(G)$. We will prove both inequalities simultaneously, since the proof is the same. Arguing by contradiction, suppose $y$ or $f:=uv$ is a counterexample, and set $Y^+=Y \cup \{y\}$ if $|X'|=1$ and $Y^+=Y \cup \{u,v\}$ if $|X'|=2$.   Then, there exists a collection $\mathcal J_1$ of more than $c_{\ref{coherence}}(j, c_{\ref{sunflower}}(j, 2g+3))$  images of $J'$ in $G$ with $X$ fixed at $Y$ and $X'$ fixed at $y$ (respectively, $X'$ fixed at $f$) such that the restrictions of these images to $J' - \mathcal P$ are all distinct.  

 By \cref{coherence},  $\mathcal J_1$ contains a coherent subfamily $\mathcal J_2$ of size at least $c_{\ref{sunflower}}(j, 2g+3)$. 
Let $\mathcal{V}$ be the collection of vertex sets of $\mathcal J_2$.  Note that by coherence, $|\mathcal V|=|\mathcal J_2| \geq c_{\ref{sunflower}}(j, 2g+3)$.   By \cref{sunflower}, $\mathcal{V}$ contains an $s$-sunflower $\mathcal F$, where $s \geq 2g+3$.  
Let $Z$ be the kernel of $\mathcal{F}$.  By construction, $Y^+ \subseteq Z$.  Let $I$ be the set of internal vertices of all $P \in \mathcal P$.  Since the restrictions of each copy of $J'$ in $\mathcal F$ to $J' - \mathcal P$ are all distinct, for all $F \in \mathcal F$ there must be a vertex $w_F \in F \setminus Z$ such that $w_F$ is the image of a vertex in $V(J') \setminus I$.  By coherence, we may assume that each $w_F$ corresponds to the same vertex in  $V(J') \setminus I$.  By \cref{menger} and \cref{claim:3paths}, there are three internally disjoint paths from $w_F$ to $Y^+$ in $G[F]$ whose ends in $Y^+$ are distinct.  For each $F \in \mathcal {F}$, let $Z_F$ be the set consisting of the first vertices of $Z$ on each of these three paths.  By coherence, we may assume $Z_F$ is the same for all $F \in \mathcal {F}$.  Thus, $G$ contains a subdivision of $K_{3, 2g+3}$.  However, this is impossible, since $K_{3, 2g+3}$ does not embed in $\Sigma$. 

It follows that $c_y, c_f \leq c_{\ref{coherence}}(j, c_{\ref{sunflower}}(j, 2g+3))$ for all $y \in V(G)$ and $f \in E(G)$. The proof is complete by summing over all possible $y \in V(G)$ if $|X'|=1$, and summing over all possible $f \in E(G)$ if $|X'|=2$.  
\end{proof}

The final ingredient we need is the following `flap reduction' lemma.

\begin{lemma} \label{lem:flapreduction}
Let $H$ be a connected graph with flap-number $k \geq 1$.  Let $A$ be a subgraph of $H$ that is \textbf{maximal} (under the subgraph relation) subject to the following conditions:
\begin{itemize}
    \item $A$ has no isolated vertices, and
    \item there exists a flap $(A,B)$ of $H$ and a set $\mathcal F$  of $k$ independent flaps in $H$ with $(A,B) \in \mathcal F$.  
\end{itemize}
Then $B^+$ has flap-number $k-1$.  Moreover, $A$ is connected and $A^+$ does not contain independent flaps $(C, D)$ and $(C',D')$ such that $V(A \cap B) \subseteq V(D \cap D')$. 
\end{lemma}

\begin{proof}
We first show that $B^+$ has flap-number at least $k-1$.  To see this, let $\mathcal F$ be a set of $k$ independent flaps in $H$ such that $(A,B) \in \mathcal F$.  Every flap $(C,D) \in \mathcal F \setminus \{(A,B)\}$ corresponds to a flap $(C, D')$ in $B^+$, unless $k=2$, $H$ is planar, and $\mathcal F=\{(A,B), (B,A)\}$. In either case, $B^+$ has flap-number at least $k-1$.  

We now prove the upper bound. Towards a contradiction, let $(C_1, D_1), \dots , (C_k, D_k)$ be independent flaps in $B^+$. 
Let $X:=V(A \cap B)$.  If $X \subseteq V(S)$ for a subgraph $S$ of $B^+$, we let $S +_X A$ be the subgraph of $H$ obtained by gluing $A$ to $S$ along $X$, and deleting the edge between the ends of $X$ in $B^+$ if the edge does not exist in $H$.   If $X$ is contained in $V(D_\ell)$ for every $\ell \in [k]$, then $(A,B),(C_1, D_1 +_X A ), \dots , (C_k, D_k +_X A)$ are $k+1$ pairwise independent flaps in $H$. Thus $X$ is not contained in $V(D_\ell)$ for some $\ell \in [k]$. 
By relabelling, we may assume $\ell=1$.  Since $(V(C_1) \setminus V(D_1)) \cap X \neq \emptyset$ and $(V(C_1) \setminus V(D_1)) \cap (V(C_i) \setminus V(D_i))=\emptyset$ for all $i > 1$, we have $X \subseteq V(D_i)$ for all $i>1$.  
 Then, $(C_1 +_X A, D_1), (C_2, D_2 +_X A), \dots  (C_k, D_k +_X A)$ are $k$ independent flaps in $H$.  Since $A$ is a proper subgraph of $C_1 +_X A$, this contradicts the maximality of $A$. 

 Finally, we show that the last sentence of the lemma holds. Suppose that $A$ is disconnected and $A_1, \dots, A_c$ are the connected components of $A$. Since $H$ is connected, $A_i$ contains a vertex of $V(A \cap B)$ for all $i \in [c]$.  Thus, $c=2$ and $A_1$ and $A_2$ each contain exactly one vertex of $V(A \cap B)$.  Since neither $A_1$ nor $A_2$ is an isolated vertex, there exist $B_1$ and $B_2$ such that $(A_1,B_1)$ and $(A_2,B_2)$  are independent flaps of $H$.  Thus, $\mathcal F \setminus \{(A,B)\} \cup \{(A_1, B_1), (A_2, B_2)\}$ is a set of $k+1$ independent flaps of $H$, which contradicts that $H$ has flap-number $k$. Similarly, $A^+$ does not contain independent flaps $(C, D)$ and $(C',D')$ such that $V(A \cap B) \subseteq V(D \cap D')$. 
\end{proof}

We call a subgraph $A$ of $H$ a \defn{half-flap} if $(A,B)$ is a flap of $H$ for some $B$.  Two half-flaps $A$ and $C$ are \defn{independent} if there exist $B$ and $D$ such that $(A,B)$ and $(C,D)$ are independent flaps.  We say that $A$ is a \defn{full-half-flap} if $A$ satisfies the conditions of \cref{lem:flapreduction}.

We stress that the condition that $A$ is maximal is essential in the definition of a full-half-flap.  To see this, let $H$ be the $2 \times k$ grid, for $k$ large.  Note that even though $H$ has many $2$-separations, the flap-number of $H$ is $2$.  Moreover, the only full-half-flaps of $H$ are $H-y$, where $y$ is one of the four degree-$2$ vertices of $H$.  In this case,  $B^+ \cong K_3$, which has flap-number $1$.  Reducing on any other flap of $H$ yields a graph that still has flap-number $2$.

We now complete the proof of the upper bound in \cref{Main}. 

\begin{theorem} \label{main}
Let $c_{\ref{main}}(h,g):=6(g+1)c_{\ref{BoundedCopies}}(h,g)c_{\ref{EppsteinCor}}(h,g)c_{\ref{lem:rootedclique}}(h, g+2)^h.$
Then for every graph $H$ with $h$ vertices  and every surface $\Sigma$ of Euler genus $g$ in which $H$ embeds,
\[
C(H,\Sigma, n) \leq I(H,\Sigma,n) \leq c_{\ref{main}}(h,g)n^{f(H)}.
\]
\end{theorem}

\begin{proof}
Let $k:=f(H)$.  Since $c_{\ref{main}}(h_1,g)\cdot c_{\ref{main}}(h_2,g) \leq c_{\ref{main}}(h,g)$ whenever $h_1+h_2=h$, we may assume that $H$ is connected by induction on $|V(H)|$. 
A \defn{reduction sequence} of $H$ is a sequence of graphs $H_k, \dots, H_j$ for some $j \leq k$, where $H_k:=H$, and for all $i > j$, $H_{i-1}:=B_i^+$, where $(A_i, B_i)$ is a flap in $H_i$ satisfying the conditions of \cref{lem:flapreduction}.  By \cref{lem:flapreduction}, every reduction sequence satisfies the following properties.   

\begin{claim} \label{reductionproperties}
 Let $H_k, \dots, H_j$ be a reduction sequence of $H$, with corresponding flaps $(A_\ell, B_\ell)$ in $H_\ell$.  Then for all $\ell \in \{k, \dots, j+1\}$,  $A_\ell$ is connected and $A_\ell^+$ does not contain independent flaps $(C, D)$ and $(C',D')$ such that $V(A_\ell) \cap V(B_\ell) \subseteq V(D) \cap V(D')$. Moreover, for all $\ell \in \{k, \dots, j\}$, $H_\ell$ has flap-number $\ell$.
\end{claim}

We now establish further properties of reduction sequences.  Let $H_k, \dots, H_j$ be a reduction sequence of $H$, with corresponding flaps $(A_\ell, B_\ell)$ in $H_\ell$. 
If $V(A_\ell) \cap V(B_\ell):=\{u,v\}$ and $uv \notin E(H_\ell)$, we declare $uv$ to be a \defn{fake edge} of $H_{\ell-1}$. An edge of $H_\ell$ is a fake edge if it is a fake edge of $H_{\ell'}$ for some $\ell' \geq \ell$, and it is a \defn{true edge} if it is not a fake edge.  

\begin{claim} \label{true3conn}
Let $H_k, \dots, H_j$ be a reduction sequence of $H$, with corresponding flaps $(A_\ell, B_\ell)$ in $H_\ell$.   Then for all $\ell \in \{k, \dots, j\}$, if $(C,D)$ is a flap in $H_\ell$ such that $C^+$ is $3$-connected, then every edge of $C$ with neither end in $V(C \cap D)$ is true.
\end{claim}

\begin{proof}
Let $(C,D)$ be a flap of $H_\ell$ that is a counterexample with $|E(C)|-|V(D)|$ maximum.  Let $X:=V(C \cap D)$. Note that, by the maximality of $|E(C)|-|V(D)|$, if $H_\ell$ contains an edge $f$ whose ends are $X$, then $f \in E(C)$.  We claim that each $x \in X$ is incident to an edge of $D$.  If not, then $x$ must be an isolated vertex of $D$.  But now, $(C, D-x)$ contradicts the maximality of $|E(C)|-|V(D)|$.  For all $i \in \{k, \dots, \ell+1\}$ set $X_i:=V(A_i \cap B_i)$.  Let $\mathcal I$ be the set of indices $i \in \{k, \dots, \ell+1\}$ such that $X_i$ is the set of ends of a fake edge of $C^-$.  Let $s$ be the smallest index in $\mathcal I$, and let $e$ be the corresponding fake edge.  Since $H_k, \dots, H_j$ is a reduction sequence, there is a collection $\mathcal F_s$ of $s$ independent flaps of $H_s$ such that $(A_s, B_s) \in \mathcal F_s$ and $\sum_{(A',B') \in \mathcal F_s}|V(A')| +|E(A')|$ is minimum.  Let $(A,B)$ be an arbitrary flap in $\mathcal F_s \setminus \{(A_s, B_s)\}$.

There are several cases to consider depending on how $A$ and $C$ interact.

Suppose that $V(A)$ is a proper subset of $V(C)$.  Since $C^+$ is $3$-connected and $V(C) \setminus V(A) \neq \emptyset$, for some $Y \in \{X_s, X\}$, one vertex $y$ of $Y$ is in $V(A) \setminus V(B)$ and the other vertex of $Y$ is in $V(B) \setminus V(A)$.  Observe that  $V(A_s) \subseteq V(B)$ because $(A_s, B_s)$ and $(A,B)$ are independent flaps of $H_s$. In particular, $X_s \subseteq V(B)$.  By the minimality of $\sum_{(A',B') \in \mathcal F_s}|V(A')| +|E(A')|$, if $H_s$ contains an edge $f$ whose ends are $X$, then $f \in E(B)$.  Since $V(A)$ is a proper subset of $V(C)$ and each $x \in X$ is incident to an edge of $D$, this implies $X \subseteq V(B)$.  Therefore, for either choice of $Y$, we have $y \in V(B)$. Thus, $y \in V(A \cap B)$, which contradicts that $y \in V(A) \setminus V(B)$.

Suppose that $V(A)=V(C)$ and $|X_s \cup X| \geq 3$. Observe that $X_s \cup X \subseteq V(A)$.  Since $V(A_s) \subseteq V(B)$, $X_s \subseteq V(B)$. Moreover, since $V(D) \subseteq V(B)$, $X \subseteq V(B)$.  Therefore, $X_s \cup X \subseteq V(A \cap B)$.  
Since $|X_s \cup X| \geq 3$, this implies that $(A,B)$ is a $(\geq 3)$-separation of $H_s$,  which contradicts that $(A,B)$ is a flap. 
  
Suppose that $A \cap C=S$, where $S$ is a stable set and $S$ contains a vertex $x \notin X$.  If $x \notin X_s$, then $x$ is an isolated vertex of $A$.  If $x \in X_s$, then since $(A,B)$ and $(A_s, B_s)$ are independent flaps, $V(A \cap A_s) \setminus X_s = \emptyset$.  Thus, $x$ is an isolated vertex of $A$ in this case as well. But now, replacing $(A,B)$ by $(A-x, B)$ contradicts the minimality of $\sum_{(A',B') \in \mathcal F_s}|V(A')| + |E(A')|$. 
  
Suppose that $A \cap C$ consists of just a single edge $xy$. Since $C^+$ is $3$-connected, $\deg_{C^+}(x) \geq 3$ and $\deg_{C^+}(y) \geq 3$.  For $ z \in \{x,y\}$, let $d(z)$ be the number of $Y \in \{X_s, X\} \setminus \{\{x,y\}\}$ such that $z \in Y$.  Observe that $\deg_{C^+}(z) \leq \deg_{A \cap C}(z) + \deg_{B \cap C} (z)+d(z)$ for both $z \in \{x,y\}$.  Since $\deg_{A \cap C}(z)=1$ for both $z \in \{x,y\}$, this yields $\deg_{B \cap C}(z) \geq 2-d(z)$ for both $z \in \{x,y\}$.  
If $\{x,y\} \subseteq X \cup X_s$, then $d(x) \leq 1 $ and $d(y) \leq 1$.  
Therefore, $\deg_{B \cap C}(x) \geq 1$ and $\deg_{B \cap C}(y) \geq 1$, 
which implies $\{x,y\} \subseteq V(A \cap B)$.  But now, $(A -xy, B \cup \{xy\})$ is a $(\leq 2)$-separation, which contradicts the minimality of  $\sum_{(A',B') \in \mathcal F_s}|V(A')| + |E(A')|$.
By symmetry, we may assume $y \notin X_s \cup X$.  Thus, $d(y)=0$, which gives $\deg_{B \cap C} (y) \geq 2$.  In particular, $y \in V(A \cap B)$.  Moreover, since $y \notin X_s \cup X$, we have $y\in V(C) \setminus V(D)$, and thus $\deg_A(y)=\deg_{A \cap C}(y)=1$.  
Observe that $|V(A - y) \cap V(B \cup \{xy\})| \leq 2$ because $V(A - y) \cap V(B \cup \{xy\}) \subseteq (V(A \cap B) \setminus \{y\}) \cup \{x\}$. Therefore
$(A -y, B \cup \{xy\})$ is a $(\leq 2)$-separation.  
But now $A-y$ is a half-flap contained in $A$,  which contradicts the minimality of $\sum_{(A',B') \in \mathcal F_s}|V(A')| + |E(A')|$. 

Suppose that $V(C)$ is a proper subset of $V(A)$.  Clearly, $X_s \subseteq V(A)$. Also, $X_s \subseteq V(B)$ since $(A,B)$ and $(A_s, B_s)$ are independent flaps.  Therefore, $V(A \cap B)=X_s$, since $|X_s|=2$.  Moreover, $A \cap D$ contains at least one vertex not in $X$, since $V(C)$ is a proper subset of $V(A)$. If $A \cap D$ meets $B \cap D$ at a vertex $x \notin X$, then $x \notin X_s$ and $x \in V(A \cap B)$, which contradicts that $V(A \cap B)=X_s$.  Thus, $V(A \cap D) \cap V(B \cap D) \subseteq X$. It follows that $A \cap D$ is a half-flap, since $|X| \leq 2$.  However, this contradicts the minimality of $\sum_{(A',B') \in \mathcal F_s}|V(A')| +|E(A')|$ since $|V(A \cap D)| < |V(A)|$.  

Suppose that $3 \leq |V(A \cap C)| < |V(C)|$ and $V(A) \setminus V(C) \neq \emptyset$.  Since $C^+$ is $3$-connected, for some $Y \in \{X_s, X\}$, one vertex $y$ of $Y$ is in $V(A) \setminus V(B)$ and the other vertex of $Y$ is in $V(B) \setminus V(A)$.
Since $(A,B)$ and $(A_s, B_s)$ are independent flaps, we have $A_s \subseteq B$.  Thus, $X_s \subseteq V(B)$, and so $Y=X$. Let $B^*$ be obtained from $B$ by replacing $A_s$ by an edge whose ends are $X_s$, and adding the edge with ends $X$.  Since $C^+$ is $3$-connected, $(C^+ \cap A, C^+ \cap B^*)$ is a $(\geq 3)$-separation of $C^+$.  
It follows that $A \cap C$ meets $B \cap C$ in at least two vertices of $V(C) \setminus X$, and thus exactly two since $|V(A \cap B)| \leq 2$. 
In particular,  $V(A \cap B) \subseteq V(C) \setminus X$, and hence  $V(A \cap D) \cap V(B \cap D) = \emptyset$.  It follows that $A \cap D$ is a half-flap.  However, this contradicts the minimality of $\sum_{(A',B') \in \mathcal F_s}|V(A')| +|E(A')|$ since $|V(A \cap D)| < |V(A)|$.

By the previous cases, there are only two cases left to consider for $(A, B)$, namely (1) $V(A \cap C)\subseteq X$ and $A \cap C$ has no edges, or (2) $V(A)=V(C)$ and $|X_s \cup X| \leq 2$. 
Since $(A,B)$ is an arbitrary flap of $\mathcal F_s \setminus \{(A_s, B_s)\}$, we may assume that either $V(A' \cap C)\subseteq X$ and $A' \cap C$ has no edges for all $(A',B') \in \mathcal F_s \setminus \{(A_s, B_s)\}$; or that $V(A)=V(C)$ and $|X_s \cup X| \leq 2$.

Suppose that $V(A' \cap C)\subseteq X$ and $A' \cap C$ has no edges for all $(A',B') \in \mathcal F_s \setminus \{(A_s, B_s)\}$. By replacing $(A_s, B_s)$ by $(A_s +_{X_s} C, B_s -(V(C) \setminus X))$ in $\mathcal F_s$, we contradict that $A_s$ is a full-half-flap.
  
The last case is that $V(A)=V(C)$ and $|X_s \cup X| \leq 2$.  Since $X_s$ is the set of ends of a fake edge of $C^-$, this implies that $X_s \neq X$. Thus, we must have $|X|=1$ and $X \subseteq X_s$.  
Let $t$ be the smallest index in $\mathcal I$ larger than $s$, and let $f$ be the corresponding fake edge.  Note that $t$ exists since $C^-$ contains a fake edge with neither end in $X$.  
Let $\mathcal F_t$ be a collection of $t$ independent flaps of $H_t$ such that $(A_t, B_t) \in \mathcal F_t$ and $\sum_{(A',B') \in \mathcal F_t}|V(A')|+|E(A')|$ is minimum.  Let $(A^*,B^*)$ be an arbitrary flap in $\mathcal F_t \setminus \{(A_t, B_t)\}$.  Let $A_s^* \subseteq H_t$ be obtained from $A_s$ by reversing all the flap reductions for all $j \in [s,t]$.   The remainder of the proof is essentially the same as the previous cases with $(A^*,B^*)$ taking the role of $(A,B)$ and $\{X_s, X_t\}$ taking the role of $\{X_s, X\}$.  For completeness, we include all the details.  

Suppose that $V(A^*)$ is  a proper subset of $V(C)$. Since $C$ is $3$-connected, for some $Y \in \{X_s, X_t\}$, $V(A^*) \setminus V(B^*)$ contains one vertex $y$ of $Y$ and $V(B^*) \setminus V(A^*)$ contains the other vertex of $Y$.   By \cref{lem:flapreduction}, $A_s^*$ and $A_t$ are both connected.  Since $V(A^*) \subseteq V(C)$, this implies $X_s \cup X_t \subseteq V(B^*)$. 
Thus, for either choice of $Y$, we have $y \in V(A^* \cap B^*)$, which contradicts that $y \in V(A^*) \setminus V(B^*)$. 

Suppose that $V(A^*)=V(C)$.  Clearly, $X_s \cup X_t \subseteq V(C)=V(A^*)$. Since $A_s^*$ and $A_t$ are both connnected by~\cref{lem:flapreduction}, $X_s \cup X_t \subseteq V(B^*)$.  Therefore, $X_s \cup X_t \subseteq V(A^* \cap B^*)$.  
Since $|X_s \cup X_t| \geq 3$, this implies that $(A^*,B^*)$ is a $(\geq 3)$-separation of $H_t$,  which contradicts that $(A^*,B^*)$ is a flap.

Suppose that $A^* \cap C=S$, where $S$ is a stable set and $S$ contains a vertex $x \notin X_s$.  If $x \notin X_t$, then $x$ is an isolated vertex of $A$.  If $x \in X_t$, then since $(A^*,B^*)$ and $(A_t, B_t)$ are independent flaps, $V(A^* \cap A_t) \setminus X_t = \emptyset$.  Thus, $x$ is an isolated vertex of $A^*$ in this case as well.  
But now, replacing $(A^*,B^*)$ by $(A^*-x, B^*)$ contradicts the minimality of $\sum_{(A',B') \in \mathcal F_t}|V(A')| + |E(A')|$. 

Suppose that $A^* \cap C$ consists of just a single edge $xy$.  Since $C$ is $3$-connected, $\deg_{C}(x) \geq 3$ and $\deg_{C}(y) \geq 3$.   For $ z \in \{x,y\}$, let $d(z)$ be the number of $Y \in \{X_s, X_t\} \setminus \{\{x,y\}\}$ such that $z \in Y$.  Observe that $\deg_{C}(z) \leq \deg_{A^* \cap C}(z) + \deg_{B^* \cap C} (z)+d(z)$ for both $z \in \{x,y\}$.  Since $\deg_{A^* \cap C}(z)=1$ for both $z \in \{x,y\}$, this yields $\deg_{B^* \cap C}(z) \geq 2-d(z)$ for both $z \in \{x,y\}$.  
If $\{x,y\} \subseteq X_s \cup X_t$, then $d(x) \leq 1 $ and $d(y) \leq 1$.  
Therefore, $\deg_{B^* \cap C}(x) \geq 1$ and $\deg_{B^* \cap C}(y) \geq 1$, 
which implies $\{x,y\} \subseteq V(A^* \cap B^*)$.  But now, $(A^* -xy, B^* \cup \{xy\})$ is a $(\leq 2)$-separation, which contradicts the minimality of  $\sum_{(A',B') \in \mathcal F_t}|V(A')| + |E(A')|$.
By symmetry, we may assume $y \notin X_s \cup X_t$.  Thus, $d(y)=0$, which gives $\deg_{B^* \cap C} (y) \geq 2$.  In particular, $y \in V(A^* \cap B^*)$.  Moreover, since $y \notin X_s \cup X_t$, we have $y\in V(C) \setminus V(D)$, and thus $\deg_{A^*}(y)=\deg_{A^* \cap C}(y)=1$.  
Observe that $|V(A^* - y) \cap V(B^* \cup \{xy\})| \leq 2$ because $V(A^* - y) \cap V(B^* \cup \{xy\}) \subseteq (V(A^* \cap B^*) \setminus \{y\}) \cup \{x\}$. Therefore
$(A^* -y, B^* \cup \{xy\})$ is a $(\leq 2)$-separation.  
But now $A^*-y$ is a half-flap contained in $A^*$,  which contradicts the minimality of $\sum_{(A',B') \in \mathcal F_t}|V(A')| + |E(A')|$.

Suppose that $V(C)$ is a proper subset of $V(A^*)$.  Clearly, $X_t \subseteq V(A^*)$. Also, $X_t \subseteq V(B^*)$ since $(A^*,B^*)$ and $(A_t, B_t)$ are independent flaps.  Therefore, $V(A^* \cap B^*)=X_t$, since $|X_t|=2$.  Moreover, $A^* \cap (D \cup A_s^*)$ contains at least one vertex not in $X_t$, since $V(C)$ is a proper subset of $V(A^*)$. If $A^* \cap (D \cup A_s^*)$ meets $B^* \cap (D \cup A_s^*)$ at a vertex $x \notin X_s$, then $x \in V(A \cap B)$, which contradicts that $V(A \cap B)=X_s$.  Thus, $V(A^* \cap (D \cup A_s^*)) \cap V(B^* \cap (D \cup A_s^*)) \subseteq X_s$. It follows that $A^* \cap (D \cup A_s^*)$ is a half-flap, since $|X_s| \leq 2$.  However, this contradicts the minimality of $\sum_{(A',B') \in \mathcal F_t}|V(A')| +|E(A')|$ since $|V(A^* \cap (D \cup A_s^*))| < |V(A)|$.

 Suppose that $3 \leq |V(A^* \cap C)| < |V(C)|$ and $V(A^*) \setminus V(C) \neq \emptyset$.  Since $C$ is $3$-connected, for some $Y \in \{X_s, X_t\}$, one vertex $y$ of $Y$ is in $V(A^*) \setminus V(B^*)$ and the other vertex of $Y$ is in $V(B^*) \setminus V(A^*)$.
Since $(A^*,B^*)$ and $(A_t, B_t)$ are independent flaps, we have $A_t \subseteq B^*$.  Thus, $X_t \subseteq V(B^*)$, and so $Y=X_s$.    
Let $\beta$ be obtained from $B^*$ by replacing $A_t$ by an edge whose ends are $X_t$ and adding the edge with ends $X_s$.   Since $C$ is $3$-connected, $(C \cap A^*, C \cap \beta)$ is a $(\geq 3)$-separation of $C$.  It follows that $A \cap C$ meets $B \cap C$ in at least two vertices of $V(C) \setminus X$, and thus exactly two since $|V(A \cap B)| \leq 2$. 

Thus, $A^* \cap C$ meets $B^* \cap C$ in at least two vertices of $V(C) \setminus X_s$, and thus exactly two since $|V(A^* \cap B^*)| \leq 2$. In particular,  $V(A^* \cap B^*) \subseteq V(C) \setminus X_s$, and hence  $V(A^* \cap (D \cup A_s^*)) \cap V(B^* \cap (D \cup A_s^*)) = \emptyset$.  It follows that $A \cap (D \cup A_s^*)$ is a half-flap.  However, this contradicts the minimality of $\sum_{(A',B') \in \mathcal F_s}|V(A')| +|E(A')|$ since $|V(A \cap D)| < |V(A)|$.

Since $(A^*,B^*)$ is an arbitrary flap of $\mathcal F_t \setminus \{(A_t, B_t)\}$, by the previous cases, we may assume  $V(A' \cap C)\subseteq X_s$ and $A' \cap C$ has no edges for all $(A',B') \in \mathcal F_t \setminus \{(A_t, B_t)\}$. 
Therefore, by replacing $(A_t, B_t)$ by $(A_t +_{X_t} C, B_t -(V(C) \setminus X_t))$ in $\mathcal F_t$, we contradict that $A_t$ is a full-half-flap.
  \end{proof}

We will need to pick reduction sequences that satisfy a few additional properties.  Let $H_k, \dots, H_j$ be a reduction sequence of $H$, with corresponding flaps $(A_\ell, B_\ell)$ in $H_\ell$.   We say that $H_k, \dots, H_j$ is a \defn{good reduction sequence} if for all $\ell \in \{k, \dots, j\}$,  $H_\ell$ is not a cycle, and for all $\ell \in \{k, \dots, j+1\}$
 
   \begin{enumerate}
       \item  \label{trueP3} if $A_\ell \cong P_3$ and $V(A_\ell) \cap V(B_\ell)=\{x\}$, then $e$ is a true edge of $H_\ell$, where $e$ is the unique edge of $A_\ell$ not incident to $x$, 
       \item  \label{trueC3} if $A_\ell \cong C_3$ and $V(A_\ell) \cap V(B_\ell)=\{x\}$, then $e$ is a true edge of $H_\ell$, where $e$ is the unique edge of $A_\ell$ not incident to $x$, 
       \item \label{trueC4} if $A_\ell \cong C_4$ and $V(A_\ell) \cap V(B_\ell)=\{x,y\}$, then $e$ is a true edge of $H_\ell$, where $e$ is the unique edge of $A_\ell$ not incident to either $x$ or $y$.
   \end{enumerate}
   
Note that, by \cref{lem:flapreduction}, if $A_\ell \cong P_3$ and $V(A_\ell) \cap V(B_\ell)=\{x\}$ then $x$ cannot be the center of the $P_3$, hence edge $e$ is well-defined in case \eqref{trueP3} above. Similarly, if $A_\ell \cong C_4$ and $V(A_\ell) \cap V(B_\ell)=\{x,y\}$, then $x$ and $y$ cannot be opposite vertices of the $C_4$, and thus $e$ is also well-defined in case \eqref{trueC4}.  

We now give conditions under which a good reduction sequence can be extended to a longer good reduction sequence. Let $H'$ be a graph where some edges are fake, and let $u, v \in V(H')$.  A \defn{$u$-culdesac} of $H'$ is a cycle $C$ of $H'$ such that $u \in V(C), \deg_{H'}(u) \geq 3$, and $\deg_{H'}(w)=2$ for all $w \in V(C) \setminus \{u\}$.   
   A \defn{$u$-alley} of $H'$ is a path $P$ of $H'$ such that $u$ is an end of $P$, $|V(P)| \geq 2, \deg_{H'}(u) \geq 3$, and $\deg_P(w)=\deg_{H'}(w)$ for all $w \in V(P) \setminus \{u\}$.  A $uv$-path $P$ in $H'$ is a \defn{$uv$-alley} if $|V(P)| \geq 3$ and $\deg_{H'}(u) \geq 3$ and $\deg_{H'}(v) \geq 3$ and $\deg_{H'}(w)=2$ for all $w \in V(P) \setminus \{u,v\}$.
%   We say that $H'$ is \defn{tame} if for all $u,v \in V(H')$, each $u$-culdesac has at most one fake edge and such an edge is incident to $u$, all fake edges of each $u$-alley of $H'$ are incident to $u$, and each $u$--$v$-alley has at most one fake edge and such an edge is incident to $u$ or $v$.

%   We note that the following property follows from tameness.
   
%   \begin{claim} \label{tamek4}
%   Let $H'$ be tame.  If $(A,B)$ is a flap of $H'$ such that $A^+=K_4^-$ and $V(A \cap B)$ contains a degree-$2$ vertex of $A^+$, then there exists a clique $X'$ of $K_4^+$ such that $|X' \cup V(A \cap B)|=3$ and $\deg_{A^+}(w)=3$, where $w$ is the unique vertex of $A^+$ not in $X' \cup V(A \cap B)$.  
%   \end{claim}
   
   \begin{claim} \label{culdesac}
   Let $C$ be a $u$-culdesac of $H'$ with $|V(C)| \geq 4$, and let $e:=uv$ be an edge of $C$ incident to $u$.  Let $P$ and $Q$ be the $3$- and $4$-vertex paths of $C$ containing $e$ and ending at $u$, respectively.  If $|V(C)|$ is even, then $P$ is a full-half flap of $H'$, and if $|V(C)|$ is odd, then $Q$ is a full-half-flap.     
   \end{claim}
   
   \begin{proof}
   Let $\mathcal F$ be a collection of $f(H')$ independent flaps in $H'$, and let $\mathcal F'$ be the collection of flaps $(F_1, F_1') \in \mathcal F$ such that $E(F_1 \cap C) \neq \emptyset$. Since $C$ contains a set of $m:=\lfloor \frac{|V(C)|}{2} \rfloor$ independent half-flaps of $H'$, we must have $|\mathcal F'| \geq m$; otherwise, $|\mathcal F|$ is not maximum.  If $\mathcal F'$ contains a flap $(F_1, F_1')$ such that $F_1 \cap C=uv$, then we replace $(F_1, F_1')$ by $(F_1 - v, F_1' \cup \{uv\})$.  Similarly, we may assume that $\mathcal F'$ does not contain a flap $(F_1, F_1')$ such that $F_1 \cap C=tu$, where $t$ is the other neighbour of $u$ in $C$.  In particular, $|E(F_1 \cap C)| \geq 2$ for all $(F_1, F_1') \in \mathcal F'$.  This implies that $|\mathcal F'|  \leq m$, and hence $|\mathcal F'|=m$.  Observe that $C$ contains a set $\mathcal H$ of $m$ independent half-flaps with $P \in \mathcal H$ if $|V(C)|$ is even, and $Q \in \mathcal H$ if $|V(C)|$ is odd.  It follows that $P$ can be extended to a full-half-flap $P'$ if $|V(C)|$ is even, and $Q$ can be extended to a full-half-flap $Q'$ if $|V(C)|$ is odd.  Since every collection of $f(H')$ independent flaps of $H'$ must contain at least $m$ flaps that use an edge of $C$, and $\bigcup_{F \in \mathcal H} F=C$, it follows that $C \cap P'=P$. If $P'$ is not contained in $C$, then $P'$ contains two independent half-flaps, which is a contradiction. Thus, $P'=P$.  The same argument gives $Q'=Q$. 
   \end{proof}
   
   The same proof also establishes the following claim for $uv$-alleys. 
   
%   \begin{claim} \label{ualley}
%   Let $P$ be a $u$-alley of $H'$ with $|V(P)| \geq 4$.  If $|V(P)|$ is even, let $P'$ be the $3$-vertex subgraph of $P$ ending at $u$. If $|V(P)|$ is odd, let $P'$ be the $4$-vertex subpath of $P$ ending at $u$.  Then $P'$ is a full-half-flap of $H'$.   
%   \end{claim}
   
   \begin{claim} \label{alley}
   Let $P$ be a $uv$-alley of $H'$ with $|V(P)| \geq 5$.  If $|V(P)|$ is even, let $P'$ be the $4$-vertex subpath of $P$ ending at $u$. If $|V(P)|$ is odd, let $P'$ be $3$-vertex subpath of $P$ ending at $u$.  Then $P'$ is a full-half-flap of $H'$.   
   \end{claim}
   
An edge $e$ is at \defn{distance $2$ from a vertex} $u$ if $e$ is not incident to $u$ and $e$ is incident to an edge that is incident to $u$.  A $u$-culdesac $C$ is \defn{tame} if $C$ has a true edge at  distance $2$ from $u$. A $uv$-alley $P$ is \defn{tame} if $P$ has a true edge at distance $2$ from $u$ or $v$.  A $u$-alley $P$ is \defn{tame} if $P$ has a true edge at distance $2$ from $u$.  Finally, we say that $H'$ is \defn{tame} if for all $u,v \in V(H')$ all $u$-culdesacs, $u$-alleys, and $uv$-alleys of $H'$ are tame.

    \begin{claim} \label{planarextension}
   Let $H_k, \dots, H_j$ be a good reduction sequence such that $j \geq 3$ and $H_j$ is tame.  Then $H_k, \dots, H_j$ can be extended to a good reduction sequence $H_k, \dots, H_{j-1}$ such that $H_{j-1}$ is tame.  \end{claim}
   
      \begin{proof}
   Suppose that $H_j$ contains a $u$-culdesac $C$ with $|V(C)| \geq 4$.  Since $H_j$ is tame, $C$ contains a true edge $e$ at distance $2$ from $u$. If $|V(C)|$ is even (respectively, odd), let $P$ be the $3$-vertex (respectively, $4$-vertex) path of $C$ such that one end of $P$ is $u$ and $e \notin E(P)$. By \cref{culdesac}, $P$ is a full-half-flap of $H_j$.  Letting $H_{j-1}$ be obtained from $H_j$ by applying \cref{lem:flapreduction} with $A=P$, we have that $H_\ell, \dots, H_{j-1}$ is a good reduction sequence and $H_{j-1}$ is tame.  Thus, we may assume that all culdesacs of $H_j$ are triangles. By \cref{alley}, we may also assume that for all $u,v \in V(H_j)$, all $uv$-alleys of $H_j$ have at most four vertices.
   
   Let $(A_j, B_j)$ be a flap of $H_j$ such that $A_j$ is a full-half-flap and let $H_{j-1}=B_j^+$.  Suppose $H_{j-1}$ is a cycle.  If $(A_j, B_j)$ is a $1$-separation, then $H_{j-1} \cong C_3$, since all culdesacs of $H_j$ are triangles.  
   Since $f(C_3)=1$, this contradicts $j \geq 3$.  If $(A_j, B_j)$ is a $2$-separation, then $H_{j-1} \in \{C_3, C_4\}$ since all alleys of $H_j$ have at most four vertices. In either case, $f(H_j)=2$, which contradicts $j \geq 3$.  
   Thus $H_{j-1}$ is not a cycle. 

   Since $H_j$ is tame, it follows that $H_k, \dots, H_{j-1}$ is a good reduction sequence. It only remains to show that $H_{j-1}$ is tame. Towards a contradiction, suppose $H_{j-1}$ contains a $u$-culdesac $C$ that is not tame.  The cases in which $H_{j-1}$ contains a $u$-alley that is not tame, or a $uv$-alley that is not tame are similar and are omitted.  We assume that $(A_j,B_j)$ is a $2$-separation (the case that $(A_j,B_j)$ is a $1$-separation is easier and is omitted). Since $H_j$ is tame, there must be an edge $xy \in E(C)$ such that $V(A_j) \cap V(B_j)=\{x,y\}$. Let $P_{xu}$ and $P_{yu}$ be the $x$--$u$ and $y$--$u$ paths in $C$ such that $V(P_{xu}) \cap V(P_{yu})=\{u\}$. 
  Since all alleys of $H_j$ have at most four vertices, $|V(P_{xu})|, |V(P_{yu})| \leq 4$.  Moreover, if $|V(P_{xu})| \in \{2, 4\}$, then adding the edge of $P_{xu}$ incident to $x$ to $A_j$ contradicts that $A_j$ is a full-half-flap.  Thus, $|V(P_{xu})|, |V(P_{yu})| \in \{1, 3\}$. In particular, this implies that $C$ is not a triangle.  
    For each fake edge $e \in E(C)$ let $\ell(e)$ be the smallest index such that $V(A_{\ell(e)}) \cap V(B_{\ell(e)})$ is equal to the set of ends of $e$. Among all fake edges of $C \setminus \{xy\}$ in $H_{j-1}$, let $f$ be such that $\ell(f)$ is smallest.  Since $C$ is not a triangle, there are two  edges of $C$ at distance $2$ from $u$  (both of which are fake).  Therefore, $f$ exists.      Let $g$ be the unique edge of $C$ such that $f \cup g$ is $P_{xu}$ or $P_{yu}$.  Then adding $g$ to $A_{\ell(f)}$ contradicts that $A_{\ell(f)}$ is a full-half flap. 
      \end{proof}

  In the case that $H$ is non-planar, we do not need $j \geq 3$ in the statement of \cref{planarextension}.  
   
   \begin{claim} \label{nonplanarextension}
   Let $H_k, \dots, H_j$ be a good reduction sequence of a non-planar graph such that $j \geq 1$ and $H_j$ is tame.  Then $H_k, \dots, H_j$ can be extended to a good reduction sequence $H_k, \dots, H_{j-1}$ such that $H_{j-1}$ is tame.  
   \end{claim}
   
   \begin{proof}
   The proof is identical to the proof of \cref{planarextension}, except for the second paragraph. Instead of using the assumption $j \geq 3$, we directly observe that $H_{j-1}$ cannot be a cycle because $H_{j-1}$ is non-planar.
   Note that if $j=1$, then $H_1$ contains a flap since it is non-planar.  Therefore, $H_0$ exists.  
   \end{proof}
   
   The final ingredient we need is the existence of a certain collection of paths in $H$. 
   Let $H_k, \dots, H_j$ be a reduction sequence of $H$ with corresponding flaps $(A_\ell, B_\ell)$ in $H_\ell$, and for each $\ell \in \{k, \dots, j\}$ let $F_\ell$ be the set of fake edges of $H_\ell$. 
   For each fake edge $e$ we define a set of indices $\mathcal I(e)$ recursively as follows. Let $\ell$ be the largest index such that $e$ is a fake edge of $H_\ell$ and recursively define $\mathcal I(e)=\{\ell+1\} \cup \bigcup_{f \in F_{\ell+1} \cap E(A_{\ell+1})}\mathcal I(f)$.
   
   \begin{claim} \label{partiallysubdivided}
   Let $H_k, \dots, H_j$ be a reduction sequence of $H$ with corresponding flaps $(A_\ell, B_\ell)$ in $H_\ell$, and for each $\ell \in \{k, \dots, j\}$ let $F_\ell$ be the set of fake edges of $H_\ell$. Then for all $\ell  \in \{k, \dots, j\}$, there is a collection of paths $\mathcal P_{\ell}=\{P_f \mid f \in F_\ell\}$ in $H$ such that for all $f \in F_\ell$, $P_f$ has the same ends as $f$, and $P_f \subseteq \bigcup_{i \in \mathcal I(f)} A_i$. Moreover, letting $H_\ell':=(H_{\ell} \setminus F_{\ell}) \cup \mathcal P_{\ell}$, we have that $(H_\ell', \mathcal P_{\ell})$ is a partially subdivided graph, $H_\ell'$ is a subgraph of $H$, and $H_\ell' / \mathcal P_\ell$ is isomorphic to $H_\ell$.  
\end{claim}

\begin{proof}
We proceed by reverse induction.  Since $H_k=H$ does not contain any fake edges, we may take $\mathcal P_k:=\emptyset$.  Suppose the claim is true for some $\ell > j$, and consider $\ell-1$.  If $|V(A_\ell) \cap V(B_\ell)|=2$ then let $\{a,b\}:=V(A_\ell) \cap V(B_\ell)$. We are done by induction if $|V(A_\ell) \cap V(B_\ell)| \leq 1$ or $e:=ab$ is a true edge of $H_{\ell-1}$.  Thus, we may assume that $e$ is a fake edge of $H_{\ell-1}$.  By the final part of \cref{lem:flapreduction}, there is path $P_{e}'$ in $A_{\ell}$ between $a$ and $b$. By induction, for each fake edge $f$ in $P_{e}'$, there is a path $P_f' \subseteq \bigcup_{i \in \mathcal I(f)} A_i$.   Moreover, note that if $f_1$ and $f_2$ are distinct fake edges of $H_{\ell-1}$, then $\mathcal{I}(f_1) \cap \mathcal I (f_2)=\emptyset$.   Therefore, replacing each fake edge $f$ of $P_{e}'$ with $P_f'$, we obtain a path $P_{e}$ contained in $\bigcup_{i \in \mathcal I(e)} A_i$. 
Every other fake edge $e'$ of $H_{\ell-1}$ is a fake edge of $H_{\ell'}$ for some $\ell'> \ell$.  By induction, for each such fake edge $e'$, there is a path $P_{e'}$ contained in $\bigcup_{i \in \mathcal I(e')} A_i$. Note that $P_{f_1}$ and $P_{f_2}$ are internally-disjoint for all distinct $f_1, f_2 \in F_{\ell-1}$, since $\mathcal{I}(f_1) \cap \mathcal I (f_2)=\emptyset$   Thus, $\mathcal P_{\ell-1}:=\{P_{f} \mid f \in F_{\ell-1}\}$ is the required set of paths.  
\end{proof}
   
We now prove the theorem in the case that $H$ is non-planar. 

 \begin{claim} \label{claim:nonplanar}
 Suppose $H$ is an $h$-vertex, non-planar graph and $H_k, \dots, H_0$ is a good reduction sequence of $H$ such that $H_\ell$ is tame for all $\ell \in \{k, \dots, 0\}$.  For each $\ell \in \{0, \dots, k\}$, let $F_\ell$ be the set of fake edges of $H_\ell$.  Then for every $n$-vertex graph $G$ embeddable in a surface of Euler genus $g$, and for all $\ell \in \{0, \dots, k\}$, there are at most
$c_{\ref{BoundedCopies}}(h,g)c_{\ref{lem:rootedclique}}(h, g+2)^\ell \cdot n^\ell$ images of $H_\ell \setminus F_\ell$ in $G$ that extend to an image of $H$ in $G$.  
 \end{claim}
 
 Before proceeding with the proof, we quickly show that \cref{claim:nonplanar} does indeed imply \cref{main} when $H$ is non-planar.  First note that $H_k, \dots, H_0$ exist by \cref{nonplanarextension}.  Next, applying \cref{claim:nonplanar} for $\ell=k$, we get that there are at most $c_{\ref{BoundedCopies}}(h,g)c_{\ref{lem:rootedclique}}(h, g+2)^k \cdot n^k \leq c_{\ref{main}}(h,g)n^k$ images of $H$ in $G$.
 
 \begin{proof}
 We proceed by induction on $\ell$.  When $\ell=0$, there are at most $c_{\ref{BoundedCopies}}(h,g)$ images of $H_0 \setminus F_0$ in $G$ that extend to an image of $H$ in $G$, by \cref{BoundedCopies}.  For the inductive step, suppose there are at most $c_{\ref{BoundedCopies}}(h,g)c_{\ref{lem:rootedclique}}(h, g+2)^\ell \cdot n^\ell$ images of $H_\ell \setminus F_\ell$ in $G$ that extend to an image of $H$ in $G$, and consider $\ell+1$.
 
 Let $\phi: V(H_\ell \setminus F_\ell) \to V(G)$ be a fixed copy of $H_\ell \setminus F_\ell$ in $G$ that extends to an image $\psi:V(H) \to V(G)$ of $H$ in $G$.  For every subgraph $S$ of $H$, let $S^\psi$ be the subgraph of $G$ obtained by restricting $\psi$ to $S$.  Let $(A_{\ell+1}, B_{\ell+1})$ be the flap in $H_{\ell+1}$ such that $H_\ell=B_{\ell+1}^+$, and let $F_{A_{\ell+1}}$ be the set of fake edges of $H_{\ell+1}$ contained in $E(A_{\ell+1})$.
  By \cref{partiallysubdivided}, there is a collection of paths $\mathcal P^\psi:=\{P_f^\psi \mid f \in F_{A_{\ell+1}}\}$ in $H^\psi$ such that for all $f \in F_{A_{\ell+1}}$, $P_f^\psi$ has the same ends as $f^\psi$ and $((A_{\ell+1} \setminus F_{A_{\ell+1}})^\psi \cup \mathcal{P}^\psi, \mathcal {P}^\psi)$ is a partially subdivided subgraph of $H^\psi$.
  
   Let $G_{\ell+1}$ be the minor of $G$ obtained by contracting all but one edge from each path in $\mathcal P^\psi$. Since $((A_{\ell+1} \setminus F_{A_{\ell+1}})^\psi \cup \mathcal P^\psi) / \mathcal P^\psi$ is isomorphic to $A_{\ell+1}$, $(A_{\ell+1} \setminus F_{A_{\ell+1}})^\psi \cup \mathcal P^\psi$ becomes an image of $A_{\ell+1}$ in $G_{\ell+1}$.

  Let $X:=V(A_{\ell+1}) \cap V(B_{\ell+1})$ and $Y:=\phi(X)$. If $X:=\{a,b\}$ and $ab$ is a fake edge of $H_{\ell+1}$, then add a handle to $\Sigma$ and use the handle to draw an edge between the vertices in $Y$ to obtain a graph $G_{\ell+1}' \supset G_{\ell+1}$ embedded in a surface of Euler genus $g+2$.  Otherwise, let $G_{\ell+1}':=G_{\ell+1}$. 
 
 Note that by \cref{reductionproperties}, \cref{true3conn}, goodness of the reduction sequence, and tameness of $H_{\ell+1}$, the conditions of \cref{lem:rootedclique} are satisfied, with $J=A_{\ell+1}^+$. Therefore, by \cref{lem:rootedclique}, there are at most $c_{\ref{lem:rootedclique}}(|V(A_{\ell+1}^+)|, g+2)n$ images of $A_{\ell+1}^+$ with $X$ rooted at $Y$ in $G_{\ell+1}'$.  Hence, there are at most $c_{\ref{lem:rootedclique}}(|V(A_1^+)|, g+2)n$ images of $H_{\ell+1} \setminus F_{\ell+1}$ in $G$ that extend $\phi$ and also extend to an image of $H$ in $G$.  By induction, there are at most $c_{\ref{BoundedCopies}}(h,g)c_{\ref{lem:rootedclique}}(h, g+2)^\ell \cdot n^\ell$ possibilities for $\phi$, so there are at most 
 $$c_{\ref{lem:rootedclique}}(|V(A_1^+)|, g+2) \cdot c_{\ref{BoundedCopies}}(h,g)c_{\ref{lem:rootedclique}}(h, g+2)^\ell \cdot n^{\ell+1} \leq c_{\ref{BoundedCopies}}(h,g)c_{\ref{lem:rootedclique}}(h, g+2)^{\ell+1} \cdot n^{\ell+1}
 $$
 images of $H_{\ell+1} \setminus F_{\ell+1}$ in $G$ that extend to an image of $H$ in $G$. 
 \end{proof}
 
 The case when $H$ is planar is similar, except that we must handle the case when $H$ is a cycle separately, which we do now.  We make a case distinction depending if the cycle has even or odd length. 
 
 Suppose $H \cong C_{2k}$.  Note that $f(C_{2k})=k$. Let $M$ be a perfect matching of $C_{2k}$.  For each copy $\phi$ of $C_{2k}$ in $G$, let $\phi(M)$ be the subset of $E(G)$ that $\phi$ maps $M$ to. Since there are at most $\binom{|E(G)|}{k}$ choices for $\phi(M)$, and each such choice corresponds to at most $2^k$ images of $C_{2k}$ in $G$, there are at most $2^k\binom{|E(G)|}{k} \leq c_{\ref{main}}(2k,g)n^k$ images of $C_{2k}$ in $G$.
 
 Suppose $H \cong C_{2k+1}$.  Note that $f(C_{2k+1})=k$.  We prove by induction on $n=|V(G)|$ that there are at most $6^k\binom{6(g+1)}{2}(g+1)^k \cdot n^k$ images of $C_{2k+1}$ in $G$.   For each vertex $a$ of $C_{2k+1}$ let $N_{C_{2k+1}}(a)$ be the two neighbours of $a$, and let $M_a$ be the unique perfect matching of $C_{2k+1}-(N_{C_{2k+1}}(a) \cup \{a\})$. Since $G$ is embedded in a surface of Euler genus $g$, $G$ has a vertex $x$ of degree at most $6(g+1)$.  
 For each copy $\phi$ of $C_{2k+1}$ in $G$ containing $x$, there are at most $\binom{|E(G-x)|}{k-1}$ choices for $\phi(M_x)$.  Since $x$ has degree at most $6(g+1)$ in $G$, there are at most $\binom{6(g+1)}{2}$ choices for $N_{C_{2k+1}}(x)$.  Each choice of $M_x$ and $N_{C_{2k+1}}(x)$ yields at most $2^k$ images of $C_{2k+1}$, so there are at most $2^k\binom{6(g+1)}{2} \binom{3(g+1)(n-1)}{k-1}$ images of $C_{2k+1}$ in $G$ containing $x$.  By induction there are at most $6^k\binom{6(g+1)}{2}(g+1)^k \cdot (n-1)^k$ images of $C_{2k+1}$ in $G-x$.  Summing these two bounds, we conclude that there are at most $6^k\binom{6(g+1)}{2}(g+1)^k \cdot n^k$ images of $C_{2k+1}$ in $G$.  Note that $6^k\binom{6(g+1)}{2}(g+1)^k \leq c_{\ref{main}}(2k+1, g)$.      

If $H$ is planar and $f(H)=1$, then there are at most $c_{\ref{EppsteinCor}}(h,g) n \leq c_{\ref{main}}(h, g)n$ images of $H$ in $G$.  Therefore, the last remaining case is when $H$ is planar, $f(H) \geq 2$, and $H$ is not a cycle.  
This is handled by the following claim.  
 
 \begin{claim} \label{edgesurvivor}
 Suppose $H$ is an $h$-vertex planar graph, $f(H) \geq 2$, and $H$ is not a cycle.  Let $H_k, \dots, H_2$ be a good reduction sequence of $H$ such that $H_\ell$ is tame for all $\ell \in \{k, \dots, 2\}$  Let $H_1:=uv$, where $uv$ is a true edge of $H_2$ such that there do not exist independent flaps $(C, D)$ and $(C',D')$ of $H_2$ such that $\{u,v\} \subseteq V(D \cap D')$. For all $\ell \in \{1, \dots, k\}$, let $F_\ell$ be the set of fake edges of $H_\ell$.  Then for every $n$-vertex graph $G$ embeddable in a surface of Euler genus $g$, and each $\ell \in \{1, \dots, k\}$, there exist at most $6(g+1)c_{\ref{lem:rootedclique}}(h, g+2)^{\ell-1} \cdot n^\ell$ images of $H_\ell \setminus F_\ell$ in $G$ that extend to an image of $H$ in $G$.  
 \end{claim}
 
 Before proceeding with the proof we note that $H_k, \dots, H_2$ exist by \cref{planarextension}. 
 The true edge $uv$ exists, by considering a leaf node of the SPQRK tree of $H_2$ and using \cref{true3conn} and the tameness of $H_2$.  Again, \cref{main} follows by taking $\ell=k$.  
 \begin{proof}
 We proceed by induction on $\ell$.  For $\ell=1$,  $H_1=uv$ is a true edge.  Therefore, there are at most $6(g+1)n$ images of $H_1$ in $G$.
 For $\ell=2$, we apply \cref{lem:rootedclique}, with $J'=(H_2 \setminus F_2) \cup \mathcal P_2, \mathcal P= \mathcal P_2, J=H_2$, and $X=\{u,v\}$, where $\mathcal P_2$ is defined in \cref{partiallysubdivided}.  
 Note that conditions (\ref{star}) and (\ref{trueX}) of \cref{lem:rootedclique} hold since $uv$ is a true edge and there do not exist independent flaps $(C, D)$ and $(C',D')$ of $H_2$ such that $\{u,v\} \subseteq V(D \cap D')$.  Conditions (\ref{smallp3})-(\ref{item:true3conn}) hold vacuously.  Conditions (\ref{p3trueflap}), (\ref{c3trueflap2}), and (\ref{c4trueflap}) hold by goodness of the reduction sequence. Condition (\ref{c3trueflap}) holds since $H_2$ is tame. Finally, (\ref{true3flap}) holds by \cref{true3conn}.  Therefore, each of the at most $6(g+1)n$ images of $H_1$ in $G$ extends to an image of $J'$ in at most $c_{\ref{lem:rootedclique}}(h, g) n$ ways.  
 Therefore, there are at most $6(g+1)c_{\ref{lem:rootedclique}}(h, g)n^2 \leq 6(g+1)c_{\ref{lem:rootedclique}}(h, g+2) \cdot n^2$ images of $H_2 \setminus F_2$ that extend to an image of $H$ in $G$.   For $\ell \geq 3$, the inductive step is exactly as in the non-planar case.  Therefore, we conclude that for each $\ell \in \{1, \dots, k\}$ there are at most $6(g+1)c_{\ref{lem:rootedclique}}(h, g+2)^{\ell-1} \cdot n^\ell$ images of $H_\ell \setminus F_\ell$ in $G$ that extend to an image of $H$ in $G$.  
 \end{proof}

This completes the proof of \cref{main}. 
\end{proof}

%%%%%%%%
\section{Copies of Complete Graphs}
\label{CompleteGraphs}

This section studies the maximum number of copies of a given complete graph $K_s$ in an $n$-vertex graph that embeds in a given surface $\Sigma$. The flap-number of $K_s$ equals $1$ if $s\leq 4$ and equals $0$ if $s\geq 5$. Thus \cref{Main} implies that $C(n,K_s,\Sigma)=\Theta(n)$ for $s\leq 4$ and $C(n,K_s,\Sigma)=\Theta(1)$ for $s\geq 5$.
The bounds obtained in this section are much more precise than those given by \cref{Main}. Our method follows that of \citet*{DFJSW}, who characterised the $n$-vertex graphs that embed in a given surface $\Sigma$ and with the maximum number of complete subgraphs (in total), and then derived an upper bound on this maximum. 

A \defn{triangulation} of a surface $\Sigma$ is an embedding of a graph in $\Sigma$ in which each facial walk has three vertices and three edges with no repetitions. Let $G$ be a triangulation of $\Sigma$. An edge $vw$ of $G$ is \defn{reducible} if $vw$ is in exactly two triangles in $G$. And $G$ is \defn{irreducible} if no edge of $G$ is reducible~\citep{BE-IJM88,BE-IJM89,CDP-CGTA04,Sulanke06,Sulanke-KleinBottle,Sulanke-Generating,LawNeg-JCTB97,NakaOta-JGT95, JoretWood-JCTB10,Lavrenchenko}. 
%Note that $K_3$ is a triangulation of $\mathbb{S}_0$, and by the above definition, $K_3$ is irreducible. In fact, it is the only irreducible triangulation of $\mathbb{S}_0$. 
%We take this somewhat non-standard approach so that ???? holds for all surfaces. 
\citet*{BE-IJM88,BE-IJM89} proved that each surface has a finite number of irreducible triangulations. For $\mathbb{S}_h$ with $h\leq 2$ and $\mathbb{N}_c$ with $c\leq 4$ the list of all irreducible triangulations is known~\citep{Lavrenchenko,Sulanke-KleinBottle,LawNeg-JCTB97,Sulanke-Generating}. In general, the best known upper bound on the number of vertices in an irreducible triangulation of a surface with Euler genus $g\geq 1$ is $13g-4$, due to \citet*{JoretWood-JCTB10}.  

Let $vw$ be a reducible edge of a triangulation $G$ of $\Sigma$.  
Let $vwx$ and $vwy$ be the two faces incident to $vw$ in $G$.  As illustrated in
\cref{ContractionSplitting}, let $G/vw$ be the graph obtained from
$G$ by \defn{contracting} $vw$; that is, delete the edges $vw,wy,wx$,
and identify $v$ and $w$ into $v$.  $G/vw$ is a simple graph since $x$
and $y$ are the only common neighbours of $v$ and $w$.  Indeed, $G/vw$
is a triangulation of $\Sigma$. Conversely, we say that $G$ is obtained from $G/vw$ by
\defn{splitting} the path $xvy$ at $v$.  If, in addition, $xy\in
E(G)$, then we say that $G$ is obtained from $G/vw$ by
\defn{splitting} the triangle $xvy$ at $v$.  Note that $xvy$ need not
be a face of $G/vw$. In the case that $xvy$ is a face, splitting $xvy$
is equivalent to adding a new vertex adjacent to each of $x,v,y$.

\begin{figure}[!ht]
\centering
\includegraphics{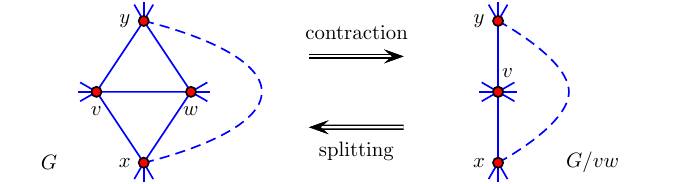}
\caption{\label{ContractionSplitting}Contracting   a reducible edge.}
\end{figure}

%%%%%%%%%%%%%%%%%%%%%%%%
\subsection{Copies of Triangles} 

In this section we consider $C(K_3,\Sigma,n)$, and define the \defn{excess} of a graph $G$ to be $C(K_3,G)-3|V(G)|$.

\begin{lemma}
\label{TriangulationK3} 
For each surface $\Sigma$, every graph embeddable in $\Sigma$ with maximum excess is a triangulation of $\Sigma$. 
\end{lemma}

\begin{proof}
Let $G$ be a graph embedded in $\Sigma$ that maximises the excess. We claim that $G$ is a triangulation. Suppose on the contrary that $F$ is a non-triangular facial walk in $G$. 

Suppose that two vertices in $F$ are not adjacent. Then there are vertices $v$ and $w$ at distance 2 in the subgraph induced by $F$. Thus adding the edge $vw$ `across' the face increases the number of triangles and the excess. This contradicts the choice of $G$. Now assume that $F$ induces a clique. 

Suppose that $F$ has at least four distinct vertices. Let $G'$ be the embedded graph obtained from $G$ by adding one new vertex `inside' the face adjacent to four distinct vertices of $F$. Thus $G'$ is embeddable in $\Sigma$, has $|V(G)|+1$ vertices, has at least $C(K_3,G)+\binom{4}{2}=C(K_3,G)+6$ triangles, and thus has excess at least the excess of $G$ plus $3$. This contradicts the choice of $G$. Now assume that $F$ has at most three distinct vertices. 

By \cref{ThreeDistinctVertices} below, $F=(u,v,w,u,v,w)$.  Let $G'$ be the graph obtained from $G$ by adding two new adjacent vertices $p$ and $q$, where $p$ is adjacent to the first $u,v,w$ sequence in $F$, and $q$ is  adjacent to the second $u,v,w$ sequence in $F$. So $G'$ is embeddable in $\Sigma$ and has $|V(G)|+2$ vertices. If $S$ is a non-empty subset of $\{p,q\}$ and $T\subseteq \{u,v,w\}$ with $|S|+|T|=3$, then $S\cup T$ is a triangle of $G'$ but not of $G$. There are $\binom{2}{1}\binom{3}{2}+\binom{2}{2}\binom{3}{1}=6+3=9$ such triangles. Thus $C(K_3,G')\geq C(K_3,G)+9$ and the excess of $G'$ is at least the excess of $G$ plus 3, which contradicts the choice of $G$. Hence no face of $G$ has repeated vertices, and $G$ is a triangulation of $\Sigma$. 
\end{proof}

\begin{lemma}
\label{ThreeDistinctVertices}
Let $F$ be a facial walk in an embedded graph, such that $F$ has exactly three distinct vertices that are pairwise adjacent. Then $F=(u,v,w)$ or $F=(u,v,w,u,v,w)$.
\end{lemma}

\begin{proof}
Say $u,v,w$ are three consecutive vertices in $F$. Then $u\neq v$ and $v\neq w$ (since there are no loops). And $u\neq w$, since if $u=w$ then $\deg(v)=1$ (since there are no parallel edges), which is not possible since $v$ is adjacent to the two other vertices in $F$. So any three consecutive vertices in $F$ are pairwise distinct. If $F$ has no repeated vertex, then $F$ is the 3-cycle $(u,v,w)$. Otherwise, $F=(u,v,w,u,\dots)$. Again, since any three consecutive vertices in $F$ are pairwise distinct, $F=(u,v,w,u,\dots)$. Repeating this argument, $F=(u,v,w,u,v,w,\dots)$. Each edge is traversed at most twice; see~\citep[Sections 3.2 and 3.3]{MoharThom}. Thus $F=(u,v,w,u,v,w)$. 
\end{proof}

\begin{theorem}
\label{ExtremalK3} 
Let $\phi$ be the maximum excess of an irreducible triangulation of $\Sigma$.  Let $X$ be the set of irreducible triangulations of $\Sigma$ with excess $\phi$.  Then the excess of every graph $G$ embeddable in $\Sigma$ is at most $\phi$. Moreover, the excess of $G$ equals $\phi$ if and only if $G$ is obtained from some graph in $X$ by repeatedly  splitting triangles. 
\end{theorem}

\begin{proof}
  We proceed by induction on $|V(G)|$.  By \cref{TriangulationK3}, we
  may assume that $G$ is a triangulation of $\Sigma$.  If $G$ is
  irreducible, then the claim follows from the definition of $X$ and
  $\phi$.  Otherwise, some edge $vw$ of $G$ is in exactly two triangles
  $vwx$ and $vwy$.  By induction, the excess of $G/vw$ is at most $\phi$. 
  Moreover, 
  %if $k\in\{3,4\}$ then 
  the excess of $G/vw$ equals $\phi$ if and only if $G$ is obtained from some graph $H\in X$ by repeatedly  splitting triangles.

  Observe that every triangle of $G$ that is not in $G/vw$ is in $\{A\cup\{w\}:A\subseteq\{x,v,y\}, |A|=2\}$. Thus $C(K_3,G)\leq C(K_3,G/vw)+3$.
  % and $f(K_4,G)\leq f(K_4,G/vw)+1$ and $f(K_k,G)=f(K_k,G/vw)$ if $k\geq 5$. 
  Moreover, 
  %if $k\in\{3,4\}$ then 
  equality holds if and only if $xvy$ is a triangle.  It follows from the definition of excess that the excess of $G$ is at most $\phi$. If the excess of $G$ equals $\phi$, then the excess of  $G/vw$ equals $\phi$, and   $xvy$ is a   triangle and $G$ is obtained from $H$ by   repeatedly splitting triangles.

  Conversely,   if
  $G$ is obtained from some $H\in X$ by
  repeatedly splitting triangles, then $xvy$ is a triangle and $G/vw$
  is obtained from $H$ by repeatedly splitting triangles.  By
  induction, the excess of $G/vw$ equals $\phi$, implying the excess of
  $G$ equals $\phi$.
\end{proof}

In general, since every irreducible triangulation of a surface $\Sigma$ with Euler genus $g$ has $O(g)$ vertices~\citep{JoretWood-JCTB10,NakaOta-JGT95}, \cref{ExtremalK3} implies that $C(K_3,\Sigma,n)\leq 3n+O(g^3)$. We now show that $C(K_3,\Sigma,n)=3n+\Theta(g^{3/2})$. 

The following elementary fact will be useful. For integers $s\geq 2$ and $m\geq 2$, 
\begin{align}
\label{NewSumReciprocals}
\sum_{i\geq m}\frac{1}{i^s} 
\leq \int_{m-1}^{\infty} i^{-s} di 
%= \frac{i^{-s+1}}{-s+1} \Big|^\infty_{m-1} 
%= \frac{(m-1)^{-s+1}}{s-1} 
= \frac{1}{(s-1)(m-1)^{s-1}}. 
\end{align}
%This implies for integers $s\geq 2$ and $m\geq 1$, 
%\begin{align}
%\label{NewNewSumReciprocals}
%\sum_{i\geq m}\frac{1}{i^s} 
%\leq \frac{1}{m^s} + \sum_{i\geq m+1}\frac{1}{i^s} 
%\leq \frac{1}{m^s} + \frac{1}{(s-1)(m)^{s-1}}
%\end{align}

\begin{theorem}
\label{CopiesK3}
For every surface $\Sigma$ of Euler genus $g$, 
$$3n+ (\sqrt{6}-o(1)) g^{3/2}  \leq C(K_3,\Sigma,n)  \leq 3n+  \frac{21}{2} g^{3/2} + O(g\log g),$$
where the lower bound holds for all $n\geq\sqrt{6g}$ and the upper bound holds for all $n$. 
\end{theorem}

\begin{proof}
First we prove the lower bound. Because of the $o(1)$ term we may assume that $g\geq 4$. Let $p:=\floor{\frac12 (7+\sqrt{24g+1})}$. Note that $p\geq 8$ and $p-\frac52>\sqrt{6g}$. The Map Colour Theorem~\citep{Ringel74} says that $K_p$ embeds in $\Sigma$. To obtain a graph with $n$ vertices embedded in $\Sigma$ repeat the following step $n-p$ times: choose a face $f$ and add a new vertex `inside' $f$ adjacent to all the vertices on the boundary of $f$. Each new vertex creates at least three new triangles. Thus 
$C(K_3,\Sigma,n)\geq 3(n-p) + \binom{p}{3}$ for $n\geq p$. 
Since $p\geq 8$ we have $\binom{p}{3} - 3p \geq \frac16(p-\frac52)^3 \geq \sqrt{6}g^{3/2}$. 
Thus $C(K_3,\Sigma,n)\geq 3n + \sqrt{6}g^{3/2}$. 

To prove the upper bound, by \cref{ExtremalK3}, it suffices to consider an $n$-vertex irreducible triangulation $G$ of $\Sigma$. So $n\leq 13g$~
\citep{JoretWood-JCTB10}. Let $v_1,\dots,v_n$ be a vertex ordering of $G$, where $v_i$ has minimum degree in $G_i:=G[\{v_1,\dots,v_i\}]$. 
By Euler's formula, $i\cdot \deg_{G_i}(v_i)  \leq 2|E(G_i)| \leq 6(i+g)$, implying $$\deg_{G_i}(v_i)\leq 6\left(1+ \frac{g}{i}\right).$$ 
Let $m:=\ceil{3\sqrt{g}}$. The number of  triangles $v_av_bv_i$ with $a<b<i\leq m$ is at most $\binom{m}{3} \leq \binom{3\sqrt{g}+1}{3} \leq \frac{9}{2} g^{3/2}$. 
Charge each triangle $v_av_bv_i$ with $a<b<i$ and $i\geq m+1$ to vertex $v_i$. 
For $m+1\leq i\leq n$, the number of triangles charged to $v_i$ is at most 
$$\binom{\deg_{G_i}(v_i)}{2}<18\left(1+\frac{g}{i}\right)^2=18\left(1+\frac{2g}{i}+\frac{g^2}{i^2}\right).$$
Thus
\begin{align*}
C(K_3,G) & \leq  
\frac{9}{2}g^{3/2} + 18\sum_{i=m+1}^{n} \left(1+\frac{2g}{i}+\frac{g^2}{i^2} \right)\\
& \leq \frac{9}{2}g^{3/2} + 18n  + 36g(\ln(n)+1) + 18g^2 \sum_{i\geq m+1}\frac{1}{i^2}.
\end{align*}
By \eqref{NewSumReciprocals} with $s=2$,
\begin{equation*}
C(K_3,G) 
\leq \frac{9}{2}g^{3/2} + 18n + 36g + 36g\ln(n) + \frac{18g^2}{m} .
\end{equation*}
Since $m\geq 3\sqrt{g}$ and $n\leq 13g$,
\begin{equation*}
C(K_3,G) 
\leq 
\frac{9}{2} g^{3/2} + 270g + 36g\ln(13g) + 6g^{3/2}
= \frac{21}{2} g^{3/2} + 270g + 36g\ln(13g).\qedhere
\end{equation*}
\end{proof}

%%%%%%%%%%%%%%%%%%%%%%%%%%%%%%%%%%%%%%%%%%
\subsection{Copies of $K_4$}

In this section, we consider the case $H=K_4$, and define the \defn{excess} of a graph $G$ to be $C(K_4,G)-|V(G)|$.

\begin{lemma}
\label{TriangulationK4} 
For each surface $\Sigma$, every graph embeddable in $\Sigma$ with maximum excess is a triangulation of $\Sigma$. 
\end{lemma}

\begin{proof}
Let $G$ be a graph embedded in $\Sigma$ with maximum excess. We claim that $G$ is a triangulation. 

Suppose that some facial walk $F$ contains non-adjacent vertices $v$ and $w$. Let $G'$ be the graph obtained from $G$ by adding the edge $vw$. Thus $C(K_4,G')\geq C(K_4,G)$. If two common neighbours of $v$ and $w$ are adjacent, then $C(K_4,G+vw)>C(K_4,G)$, implying that the excess of $G+vw$ is greater than the excess of $G$, which contradicts the choice of $G$. Now assume that no two common neighbours of $v$ and $w$ are adjacent. Let $G'':=G'/vw$. Every $K_4$ subgraph in $G'$ is also in $G''$. Thus $C(K_4,G'')\geq C(K_4,G')\geq C(K_4,G)$. Since $|V(G'')|<|V(G)|$, the excess of $G''$ is greater than the excess of $G$, which contradicts the choice of $G$. Now assume that every facial walk induces a clique in $G$. 

Suppose that some facial walk $F$ has at least four distinct vertices. Let $G'$ be the embedded graph obtained from $G$ by adding one new vertex `inside' the face adjacent to four distinct vertices of $F$. Thus $G'$ is embeddable in $\Sigma$, has $|V(G)|+1$ vertices, has at least $C(K_4,G)+\binom{4}{3}=C(K_4,G)+4$ triangles, and thus has excess at least the excess of $G$ plus $3$. This contradicts the choice of $G$. Now assume that every facial walk in $G$ has at most three distinct vertices. 

Suppose that some facial walk $F$ is not a triangle. By \cref{ThreeDistinctVertices}, $F=(u,v,w,u,v,w)$.  Let $G'$ be the graph obtained from $G$ by adding two new adjacent vertices $p$ and $q$, where $p$ is adjacent to the first $u,v,w$ sequence in $F$, and $q$ is  adjacent to the second $u,v,w$ sequence in $F$. So $G'$ is embeddable in $\Sigma$ and has $|V(G)|+2$ vertices. If $S$ is a non-empty subset of $\{p,q\}$ and $T\subseteq \{u,v,w\}$ with $|S|+|T|=4$, then $S\cup T$ induces a copy of $K_4$ in $G'$ but not in $G$. There are $\binom{2}{2}\binom{3}{2}+\binom{2}{1}\binom{3}{3}=3+2=5$ such copies. Thus $C(K_4,G')\geq C(K_4,G)+5$ and the excess of $G'$ is at least the excess of $G$ plus 3, which contradicts the choice of $G$. Therefore $G$ is a triangulation of $\Sigma$. 
\end{proof}

\begin{theorem}
\label{ExtremalK4} 
Let $\phi$ be the maximum excess of an irreducible triangulation of $\Sigma$.  Let $X$ be the set of irreducible triangulations of $\Sigma$ with excess $\phi$.  Then the excess of every graph $G$ embeddable in $\Sigma$ is at most $\phi$. Moreover, the excess of $G$ equals $\phi$ if and only if $G$ is obtained from some graph in $X$ by repeatedly  splitting triangles. 
\end{theorem}

\begin{proof}
We proceed by induction on $|V(G)|$.  By \cref{TriangulationK4}, we may assume that $G$ is a triangulation of $\Sigma$.  If $G$ is irreducible, then the claim follows from the definition of $X$ and $\phi$.  Otherwise, some edge $vw$ of $G$ is in exactly two triangles $vwx$ and $vwy$.  By induction, the excess of $G/vw$ is at most $\phi$.  Moreover, the excess of $G/vw$ equals $\phi$ if and only if $G$ is obtained from some graph $H\in X$ by repeatedly  splitting triangles.

Observe that every clique of $G$ that is not in $G/vw$ is in $\{A\cup\{w\}:A\subseteq\{x,v,y\}\}$. Thus $C(K_4,G)\leq C(K_4,G/vw)+1$. Moreover, equality holds if and only if $xvy$ is a triangle. It follows from the definition of excess that the excess of $G$ is at most $\phi$. If the excess of $G$ equals $\phi$, then the excess of  $G/vw$ equals $\phi$, and $xvy$ is a   triangle, and $G$ is obtained from $H$ by   repeatedly splitting triangles.

Conversely, if $G$ is obtained from some $H\in X$ by repeatedly splitting triangles, then $xvy$ is a triangle and $G/vw$ is obtained from $H$ by repeatedly splitting triangles.  By induction, the excess of $G/vw$ equals $\phi$, implying the excess of $G$ equals $\phi$.
\end{proof}

Since every irreducible triangulation of a surface $\Sigma$ with Euler genus $g$ has $O(g)$ vertices~\citep{NakaOta-JGT95,JoretWood-JCTB10}, \cref{ExtremalK4} implies that $C(K_4,\Sigma,n)\leq n+O(g^4)$. We now show that $C(K_4,\Sigma,n)=n+\Theta(g^{2})$. 

\begin{theorem}
\label{CopiesK4}
For every surface $\Sigma$ of Euler genus $g$, 
$$n+\frac32 g^{2} \leq C(K_4,\Sigma,n)  \leq n +
 \frac{283}{24}g^2 + O(g^{3/2}),$$
where the lower bound holds for $g\geq 1$ and $n\geq\sqrt{6g}$, and the upper bound holds for all $n$. 
\end{theorem}

\begin{proof}
First we prove the lower bound. If $\Sigma=\mathbb{N}_2$ then let $p:=6$. Otherwise, let $p:=\floor{\frac12 (7+\sqrt{24g+1})}$. Since $g\geq 1$ we have $p\geq 6$. The Map Colour Theorem~\citep{Ringel74} says that $K_p$ embeds in $\Sigma$. To obtain a graph with $n$ vertices embedded in $\Sigma$ repeat the following step $n-p$ times: choose a face $f$ and add a new vertex `inside' $f$ adjacent to all the vertices on the boundary of $f$. Each new vertex creates at least one new copy of $K_4$ (since the boundary of each face is always a clique on at least three vertices). Thus $C(K_4,\Sigma,n)\geq n-p + \binom{p}{4}$ for $n\geq p$. 
Since $\binom{p}{4}-p\geq \frac{1}{24}(p-\frac52)^4$ and $p-\frac52>\sqrt{6g}$ we have
$C(K_4,\Sigma,n)\geq n + \frac{1}{24}(\sqrt{6g})^4=n + \frac32 g^2$. 

Now we prove the upper bound. The claim is trivial for $g=0$, so now assume that $g\geq 1$. By \cref{ExtremalK4}, it suffices to consider an irreducible triangulation $G$. \citet*{JoretWood-JCTB10} proved that $n:=|V(G)|\leq 13g$. Let $v_1,\dots,v_n$ be a vertex ordering of $G$, where $v_i$ has minimum degree in $G_i:=G[\{v_1,\dots,v_i\}]$. 
By Euler's formula, 
$$i\cdot \deg_{G_i}(v_i)  \leq 2|E(G_i)| \leq 6(i+g),$$
and
$$\deg_{G_i}(v_i)  \leq 6\left(1+\frac{g}{i}\right).$$
Define $m:=\ceil{4\sqrt{g}}$. 
The number of copies $v_av_bv_cv_i$ with $a<b<c<i\leq m$ is at most $\binom{m}{4} \leq \binom{4\sqrt{g}+1}{4} \leq \frac{32}{3}g^2$. 
Charge each copy $v_av_bv_cv_i$ with $a<b<c<i$ and $i\geq m+1$ to vertex $v_i$. 
For $m+1\leq i\leq n$, the number of copies charged to $v_i$ is at most 
$$\binom{\deg_{G_i}(v_i)}{3}<36\left(1+\frac{g}{i}\right)^3 
= 36 \left( \left(\frac{g}{i}\right)^3 + 3 \left(\frac{g}{i}\right)^2 + 3 \left(\frac{g}{i}\right) + 1 \right).$$
In total,
$$C(K_4,G) \leq \frac{32}{3}g^2 
+ 36 \sum_{i=m+1}^n \left(\frac{g}{i}\right)^3 + 3 \left(\frac{g}{i}\right)^2 + 3 \left(\frac{g}{i}\right) + 1. $$
By \eqref{NewSumReciprocals} with $s=2$ and $s=3$, 
\begin{equation*}
C(K_4,G) \leq 
 \frac{32}{3}g^2 + 36 \left( \frac{g^3}{2m^2} + \frac{3g^2}{m} + 3g( \ln n + 1) + n \right).
\end{equation*}
Since $m\geq 4\sqrt{g}$ and $n\leq 13g$, 
\begin{align*}
C(K_4,G) 
& \leq 
 \frac{32}{3}g^2 + 
36 \left( \frac{g^2}{32} + \frac{3g^{3/2}}{4} + 3g( \ln (13g) + 1) + 13g \right)\\
& = 
 \frac{283}{24}g^2 + 27 g^{3/2} + 108 g( \ln (13g) + 1) + 468 g. \qedhere
\end{align*}
\end{proof}

%%%%%%%%%%%
\subsection{General Complete Graph}

Now consider the case when $H=K_s$ for some $s\ge 5$. \cref{Main} shows that $C(K_s,\Sigma,n)$ is bounded for fixed $s$ and $\Sigma$. We now show how to determine $C(K_s,\Sigma,n)$ more  precisely. 

\begin{theorem}
\label{ExtremalCompleteGraph} 
For every integer $s\geq5$ and surface $\Sigma$ there is an irreducible triangulation $G$ such that $C(K_s,G)=\max_n C(K_s,\Sigma,n)$.
\end{theorem}

\begin{proof}
Let $q:=\max_n C(K_s,\Sigma,n)$. Let $G_0$ be a graph embedded in $\Sigma$ with $C(K_s,G_0)=q$. As described in the proof of \cref{TriangulationK3} we can add edges and vertices to $G_0$ to create a triangulation $G$ of $\Sigma$. Adding edges and vertices does not remove copies of $K_s$. Thus $C(K_s,G)=q$. If $G$ is irreducible, then we are done. Otherwise, some edge $vw$ of $G$ is in exactly two triangles $vwx$ and $vwy$. Let $G':=G/vw$. Then $G'$ is another triangulation of $\Sigma$. Observe that every clique of $G$ that is not in $G'$ is in $\{A\cup\{w\}:A\subseteq\{x,v,y\}\}$. Each such clique has at most four vertices. Thus $C(K_s,G')=C(K_s,G)=q$. Repeat this step to $G'$ until we obtain an irreducible triangulation $G''$ with $C(K_s,G'')=q$.
\end{proof}

%%%%
We now prove a precise bound on $C(K_s,\Sigma,n)$, making no effort to optimise the constant 300. 

\begin{theorem}
\label{CopiesKs}
For every integer $s\geq 5$ and surface $\Sigma$ of Euler genus $g$ and for all $n$, 
$$\left(\frac{\sqrt{6 g}}{s}\right)^s 
\leq C(K_s,\Sigma,n)  \leq 
%\Big(  \left(\frac{e}{s}\right)^s  + 2 %\left(\frac{84e}{s-1}\right)^{s-1}  \Big) g^{s/2},$$
 \left(\frac{300 \sqrt{g} }{s}\right)^s,$$
where the lower bound holds for all $n\geq\sqrt{6g}\geq s$ and the upper bound holds for all $n$. 
\end{theorem}

\begin{proof}
For the lower bound, it follows from the Map Colour Theorem~\citep{Ringel74} that $K_p$ embeds in $\Sigma$ where $p:=\ceil{\sqrt{6g}}$. Thus, for $n\geq p\geq s$,  
$$C(K_s,\Sigma,n)\geq \binom{\sqrt{6g}}{s}\geq\left(\frac{\sqrt{6g}}{s}\right)^s.$$
Now we prove the upper bound. The claim is trivial for $g=0$, so assume that $g\geq 1$. By \cref{ExtremalCompleteGraph}, it suffices to consider an irreducible triangulation $G$ of $\Sigma$. \citet*{JoretWood-JCTB10} proved that $n:=|V(G)|\leq 13g$. Let $v_1,\dots,v_n$ be a vertex ordering of $G$, where $v_i$ has minimum degree in $G_i:=G[\{v_1,\dots,v_i\}]$. 
By Euler's formula, 
$$i\cdot \deg_{G_i}(v_i)  \leq 2|E(G_i)| \leq 6(i+g)\leq 6(n+g)\leq 84 g.$$
Define $m:=\ceil{\sqrt{g}}$. The number of  copies of $K_s$ in $G[\{v_1,\dots,v_m\}]$ is at most 
$$\binom{m}{s}\leq \left(\frac{2e \sqrt{g} }{s}\right)^s\leq \left(\frac{2e}{s}\right)^s g^{s/2}.$$
Charge every other copy $X$ of $K_s$ to the rightmost vertex in $X$ (with respect to $v_1,\dots,v_n$). 
For $m+1\leq i\leq n$, the number of copies of $K_s$ charged to $v_i$ is at most 
\begin{equation*}
\binom{\deg_{G_i}(v_i)}{s-1}
\leq \left( \frac{e\deg_{G_i}(v_i)}{s-1} \right)^{s-1}
\leq \left(\frac{84eg}{i(s-1)}\right)^{s-1}.
\end{equation*}
In total,
\begin{equation*}
C(K_s,G)\leq 
\left(\frac{2e}{s}\right)^s g^{s/2} + 
\left(\frac{84eg}{s-1}\right)^{s-1}
\sum_{i\geq m+1}\frac{1}{i^{s-1}}.
%\leq \frac{1}{(s-2)(m)^{s-2}}. 
\end{equation*}
%
%\sum_{i\geq m}\frac{1}{i^s} 
%\leq \int_{m-1}^{\infty} i^{-s} di 
%= \frac{1}{(s-1)(m-1)^{s-1}}. 
%
By \eqref{NewSumReciprocals}, 
\begin{equation*}
C(K_s,G)
\leq  \left(\frac{2e}{s}\right)^s g^{s/2} +  \left(\frac{84eg}{s-1}\right)^{s-1}  \frac{1}{(s-2)m^{s-2}} .
\end{equation*}
Since $m\geq\sqrt{g}$, 
\begin{equation*}
 C(K_s,G)
 \leq  \left(\frac{2e}{s}\right)^s g^{s/2} +  \left(\frac{84eg}{s-1}\right)^{s-1}  \!\!\!
 \frac{1}{(s-2)\,g^{(s-2)/2}} 
 \leq 
\left(\frac{300\sqrt{g}}{s}\right)^s.
\hfill\qedhere
\end{equation*}
\end{proof} 

% {\david The final inequality is proved as follows:
% \begin{align*}
%   &\left(\frac{2e}{s}\right)^s g^{s/2} +  \left(\frac{84eg}{s-1}\right)^{s-1}  \frac{1}{(s-2)\,g^{(s-2)/2}} 
%  & \leq 
% \left(\frac{300\sqrt{g}}{s}\right)^s\\
% %%%%%
% \Longleftarrow \quad &\left(\frac{2e}{s}\right)^s  +  \left(\frac{84e}{s-1}\right)^{s-1}  \frac{1}{s-2} 
%  & \leq 
% \left(\frac{300}{s}\right)^s\\
% %%%%%
% \Longleftarrow \quad &\left(2e\right)^s  +  s^s \left(\frac{84e}{s-1}\right)^{s-1}  \frac{1}{s-2} 
%  & \leq 
% \left(300\right)^s\\
% %%%%%
% \Longleftarrow \quad & s^s \left(\frac{84e}{s-1}\right)^{s-1}  \frac{1}{s-2} 
%  & \leq 
% 290^s\\
% %%%%%
% \Longleftarrow\quad & s \times s^{s-1} 
%   \left(\frac{84e}{s-1}\right)^{s-1}  \frac{1}{s-2} 
%  & \leq 
% 290 \times 290^{s-1}\\
% %%%%%
% \Longleftarrow \quad & \frac{s}{s-2} \times (\frac54 (s-1))^{s-1} \left(\frac{84e}{s-1}\right)^{s-1} 
%  & \leq 
% 290 \times 290^{s-1}\\
% %%%%%
% \Longleftarrow\quad  & \frac{s}{s-2} ( 105 e)^{s-1}  
%  & \leq 
% 290 \times 290^{s-1}\\
% \end{align*}
% }

\subsection{Computational Results}

For $\Sigma\in\{\mathbb{S}_0,\mathbb{S}_1, \mathbb{S}_2,\mathbb{N}_1, \mathbb{N}_2, \mathbb{N}_3,\mathbb{N}_4\}$, we use \cref{TriangulationK3,TriangulationK4,ExtremalCompleteGraph}, the lists of all irreducible triangulations~\citep{Lavrenchenko,Sulanke-KleinBottle,LawNeg-JCTB97,Sulanke-Generating}, and an elementary computer program to count cliques to obtain the exact results for $C(K_s,\Sigma,n)$ shown in \cref{Exact}.

\begin{table}[H]
\caption{\label{Exact}The maximum number of copies of $K_s$ in an $n$-vertex graph embeddable in surface $\Sigma$.}
\medskip
\begin{tabular}{c|ccccccccc|c}
\hline
 $\Sigma$ & $s=0$ & $s=1$ & $s=2$ & $s=3$ & $s=4$ & $s=5$ & $s=6$ & $s=7$ & $s=8$ & total \\
\hline
$\mathbb{S}_0$ & 1 & $n$ & $3n-6$  & $3n-8$ & $n-3$ & & & & & $8n-16$ \\
$\mathbb{S}_1$ & 1 & $n$ & $3n$     & $3n+14$ & $n+28$ & $21$ & $7$ & $1$ & & $8n+72$\\
$\mathbb{S}_2$ & 1 & $n$ & $3n+6$ & $3n+38$ & $n+68$ & $58$ & $28$ & $8$ & $1$ & $8n+208$ \\
$\mathbb{N}_1$ & 1 & $n$ & $3n-3$ & $3n+2$ & $n+9$ & $6$ & $1$ & & & $8n+16$ \\
$\mathbb{N}_2$ & 1 & $n$ & $3n$ & $3n+12$ & $n+21$ & $12$ & $2$ & & & $8n+48$\\
$\mathbb{N}_3$ & 1 & $n$ & $3n+3$ & $3n+24$ & $n+40$ & $27$ & $8$ & $1$ & & $8n+104$ \\
$\mathbb{N}_4$ & 1 & $n$ & $3n+6$ & $3n+39$ & $n+71$ & $61$ & $29$ & $8$ & $1$ & $8n+216$\\
\hline
\end{tabular}
\end{table}

Let $C(G)$ be the total number of complete subgraphs in a graph $G$; that is $C(G)=\sum_{s\geq 0}C(K_s,G)$. For a surface $\Sigma$, let $C(\Sigma,n)$ be the maximum of $C(G)$ taken over all $n$-vertex graphs $G$ embeddable in $\Sigma$. \citet*{DFJSW} proved that $C(\Sigma,n)-8n$ is bounded for fixed $\Sigma$, which is implied by \cref{CopiesK3,CopiesK4,CopiesKs}. The following conjectures have been verified for each of 
$\mathbb{S}_0$, $\mathbb{S}_1$, $\mathbb{S}_2$, $\mathbb{N}_1$, $\mathbb{N}_2$, $\mathbb{N}_3$, $\mathbb{N}_4$. 

\begin{conjecture}
For every surface $\Sigma$ and integer $n$, 
$$C(\Sigma,n)=\sum_{s\geq 0} C(K_s,\Sigma,n).$$
\end{conjecture}

\begin{conjecture}
If $C(G)=C(\Sigma,n)$ for some $n$-vertex graph $G$ embeddable in a surface $\Sigma$, then for $s\geq 0$, 
$$C(K_s,G) = C(K_s,\Sigma,n).$$
\end{conjecture}

Conversely, we conjecture that maximising the number of triangles is equivalent to maximising the total number of complete subgraphs. More precisely:

\begin{conjecture}
\label{DeterminedByTriangles}
If $C(K_3,G)=C(K_3,\Sigma,n)$ for some $n$-vertex graph $G$ embeddable in a surface $\Sigma$, then 
$$C(G) = C(\Sigma,n).$$
\end{conjecture}

Note that $K_3$ cannot be replaced by some arbitrary complete graph in \cref{DeterminedByTriangles}. For example, every graph embeddable in $\mathbb{N}_3$ contains at most one copy of $K_7$, but there are irreducible triangulations $G$ of $\mathbb{N}_3$ that contain $K_7$ and do not maximise the total number of cliques (that is, $C(G)<8|V(G)|+104$). Similarly, every graph embeddable in $\mathbb{N}_4$ contains at most $8$ copies of $K_7$, but there are irreducible triangulations $G$ of $\mathbb{N}_4$ for which $C(K_7,G)=8$ and $C(G)<8|V(G)|+216$. 

\section{Minor-Closed Classes}
\label{MinorClosedClasses}

Consider the following natural open problem extending our results for graphs on surfaces: For graphs $H$ and $X$ and an integer $n$, what is the maximum number of copies of $H$ in an $n$-vertex $X$-minor-free graph? This problem has been extensively studied when $H$ and $X$ are complete graphs~\citep{Wood16,FOT10,LO15,FW17,NSTW06,RW09}. \citet*{Eppstein93} proved the following result when $X$ is a complete bipartite graph and $H$ is highly connected. 

\begin{theorem}[\cite{Eppstein93}] 
\label{flopnumber0}
Fix positive integers $s\leq t$ and a $K_{s,t}$-minor-free graph $H$ with no $(\leq s-1)$-separation. Then every $n$-vertex $K_{s,t}$-minor-free graph contains $O(n)$ copies of $H$.
\end{theorem}

What happens when $H$ is not highly connected? We have the following lower bound. Fix positive integers $s\leq t$ and a $K_{s,t}$-minor-free graph  $H$. If $H$ has no $(\leq s-1)$-separation, then let $k:=1$; otherwise, let $k$ be the maximum number of pairwise independent $(\leq s-1)$-separations in $H$. The construction in \cref{LowerBound} generalises to give $n$-vertex $K_{s,t}$-minor-free graphs containing $\Theta(n^k)$ copies of $H$. 

The following question naturally arises: 

\begin{open}
\label{open1}
Does every $n$-vertex $K_{s,t}$-minor-free graph contain $O(n^k)$ copies of $H$? 
\end{open}

By \cref{flopnumber0}, the answer is `yes' if $k=1$. The methods presented in this paper show the answer is `yes' if $s\leq 3$. We omit the proof, since it is essentially the same as for graphs embedded on a surface, except that in the $k=1$ case we use \cref{flopnumber0} instead of the additivity of Euler genus (\cref{Additivity}).

When $H$ is a tree, this problem specialises as follows:  Fix a tree $T$ and positive integers $s\leq t$. Let $\beta(T)$ be the size of the largest stable set of vertices in $T$, each with degree at most $s-1$. The construction in \cref{TreeCopies} generalises to give $n$-vertex $K_{s,t}$-minor-free graphs containing $\Omega(n^{\beta(T)})$ copies of $T$. 

\begin{open}
\label{open2}
Does every $n$-vertex $K_{s,t}$-minor-free graph contain  $O(n^{\beta(T)})$ copies of $T$?
\end{open}

\section{Homomorphism Inequalities}
\label{Homomorphism}

This section reinterprets some of our results in terms of homomorphism inequalities, and presents some open problems that arise from this viewpoint. 

For two graphs $H$ and $G$, a \defn{homomorphism} from $H$ to $G$ is a function $\phi:V(H)\rightarrow V(G)$ that preserves adjacency; that is, $\phi(v)\phi(w)$ is an edge of $G$ for each edge $vw$ of $H$. Let $\hom(H,G)$ be the number of homomorphisms from $H$ to $G$. For example, $\hom(H,K_t)>0$ if and only if $H$ is $t$-colourable. In the other direction, $\hom(K_1, G)$ is the number of vertices in $G$, and $\hom(K_2, G)$ is twice the number of edges in $G$, and $\hom(K_3,G)$ is 6 times the number of triangles in $G$. 

Homomorphism inequalities encode bounds on the number of copies of given graphs in a host graph. Much of extremal graph theory can be written in terms of homomorphism inequalities, and a beautiful theory has recently developed that greatly simplifies the task of proving such inequalities; see~\citep{Lovasz12}. 

Consider the following concrete example. \citet*{Mantel07} proved that every $n$-vertex graph with more than $\frac{n^2}{4}$ edges has a triangle, which is tight for the complete bipartite graph $K_{n/2,n/2}$. \citet*{Goodman59} strengthened Mantel's Theorem by providing a lower bound of $\frac{m}{3} ( \frac{4m}{n} - n )$ on the number of triangles in an $n$-vertex $m$-edge graph. Goodman's Theorem can be rewritten as the following homomorphism inequality:
\begin{equation}
\label{Goodman}
\hom(K_1, G)\hom(K_3, G) \geq \hom(K_2, G)(2\hom(K_2, G) - \hom(K_1, G)^2).
\end{equation}
In a celebrated application of the flag algebra method, \citet*{Razborov08} generalised \eqref{Goodman} by determining the minimum number of triangles in an $n$-vertex $m$-edge graph. The minimum number of copies of $K_r$ in an $n$-vertex $m$-edge graph (the natural extension of Turan's Theorem) was a notoriously difficult question~\citep{LovSim76,LovSim83}, recently solved for $r=4$ by \citet*{Nikiforov11} and in general by \citet*{Reiher16}. All of these results can be written in terms of homomorphism inequalities. 

The results of this paper can be written in terms of homomorphism counts, since $I(H,G)$ equals the number of injective homomorphisms from $H$ to $G$, which equals $\hom(H,G)$ if $H$ is a complete graph. 

%show that for every fixed graph $H$ with flap-number $k$, and for every graph $G$ that embeds in a fixed surface $\Sigma$,  $$\hom(H,G) \leq c_1 \hom(K_1,G)^k;$$ and if $H$ embeds in $\Sigma$, then $\hom(H,G) \geq c_2 \hom(K_1,G)^k$ for infinitely many graphs $G$ that also embed in $\Sigma$.

Here is another example of a homomorphism inequality for graphs on surfaces. Euler's Formula implies\footnote{Let $G$ be a
graph with $n$ vertices, $m$ edges  and $c$ components. Let $\Sigma$ be a surface with Euler genus $g$. 
Assume that $G$ embeds in $\Sigma$ with $t$ triangular faces and $f$ non-triangular faces. By Euler's formula, $n-m+t+f=1+c-g$. Double-counting edges, $3t+4f \leq 2m$. Thus 
$4(m-n-t + 1+c-g) = 4f \leq 2m -3t$ and
$t \geq 2m-4n + 4+4c - 4g
\geq 2(m-2n + 4 - 2g)$, as claimed.} that the number of triangles in an $n$-vertex $m$-edge graph with Euler genus $g$ is at least $2(m-2n+4-2g)$. This result is an analogue of Goodman's Theorem for graphs $G$ of Euler genus $g$, and can be written as the following homomorphism inequality:
\begin{equation*}
\hom(K_3,G) \geq 6\hom(K_2,G)-24\hom(K_1,G) + 48 - 24g.
\end{equation*}

We consider it an interesting line of research to prove similar homomorphism inequalities in other minor-closed classes. The following open problems naturally arise. 

\begin{itemize}
    \item Is there a method (akin to flag algebras~\citep{Razborov08} or graph algebras~\citep{Lovasz12}) for systematically proving homomorphism inequalities in minor-closed classes? 
    
\item  \citet*{HN11} proved that it is undecidable to test the validity of a linear homomorphism inequality. 
In which minor-closed classes is it decidable to test  the validity of a linear homomorphism inequality? 
\end{itemize}

These questions are open even for forests; see~\citep{BEMS16,BL16,CSW17,CKMW21} for related results. 

Closely related to the study of graph homomorphisms is the theory of graph limits and graphons~\citep{Lovasz12}. While this theory focuses on dense graphs, a theory of graph limits for sparse graphs is emerging. For example, results are known for bounded degree graphs~\citep{BCKL13,HLS14},  planar graphs~\citep{IO01,GN13}, and bounded tree-depth graphs~\citep{NO20}. The above questions regarding graph homomorphisms parallel the theory of graph limits in sparse classes. 

%BCCZ14a,BCCZ14b

%A major development in extremal graph theory is the theory of flag algebras due to Razborov~\cite{Razborov07}, which is closely related to graph algebras~\cite{Lovasz12}. Flag algebras provide  tools for systematically proving homomorphism inequalities such as \eqref{Goodman}. They were crucial in Razborov~\cite{Razborov08} determining the minimum number of triangles in an $n$-vertex $m$-edge graph. 

%which is currently being developed by serval groups of researchers.  In particular, results are known for bounded degree graphs \cite{BCKL13,HLS14},  graphs with a power law degree distributions \cite{BCCZ}, bounded tree-depth graphs \cite{NO13}, and planar graphs \cite{IO01,GN13}, among others. 
% OLDAs mentioned earlier,  a major development in extremal graph theory is the notion of graph limits and graphons~\cite{LS06,Lovasz12}. This theory focuses on dense graphs. An emerging theory of graph limits for sparse graphs is currently a hot topic. Results are known for bounded degree graphs \cite{BCKL13,HLS14},  graphs with a power law degree distributions \cite{BCCZ}, bounded tree-depth graphs \cite{NO13}, and planar graphs \cite{IO01,GN13}, among others. 
%
%\smallskip\emph{Goal. Develop a theory of graph limits for minor-closed classes.}
%\smallskip 

%The starting point for this goal is the work on graph algebras and flag algebras, along with recent developments in limit theory for sparse graphs \cite{BCKL13,GN13,HLS14,IO01,NO13}. Other results in this theory lead to interesting questions for minor-closed classes. For example,

\subsection*{Acknowledgement}
We would like to thank Kevin Hendrey for alerting us to an error in the proof of \cref{main} in an earlier version of this paper.  We also thank Casey Tompkins for pointing out reference~\citep{GPSTZc}. \citet*{GPSTZc} prove \cref{TreeCopies} in the case $\Sigma=\mathbb{S}_0$, and conjecture that $C(H,\Sigma_0,n)=\Theta(n^k)$ for some integer $k=k(H)$, which is implied by \cref{Main}. Finally, many thanks to the referee for several insightful comments. 

\subsection*{Postscript}

Following the initial release of this paper, \cref{open1} was answered in the negative by \citet{Liu21} and \cref{open2} was answered in the affirmative by \citet{HW21}.

\bibliographystyle{DavidNatbibStyle}
\bibliography{paper}

\def\soft#1{\leavevmode\setbox0=\hbox{h}\dimen7=\ht0\advance \dimen7
  by-1ex\relax\if t#1\relax\rlap{\raise.6\dimen7
  \hbox{\kern.3ex\char'47}}#1\relax\else\if T#1\relax
  \rlap{\raise.5\dimen7\hbox{\kern1.3ex\char'47}}#1\relax \else\if
  d#1\relax\rlap{\raise.5\dimen7\hbox{\kern.9ex \char'47}}#1\relax\else\if
  D#1\relax\rlap{\raise.5\dimen7 \hbox{\kern1.4ex\char'47}}#1\relax\else\if
  l#1\relax \rlap{\raise.5\dimen7\hbox{\kern.4ex\char'47}}#1\relax \else\if
  L#1\relax\rlap{\raise.5\dimen7\hbox{\kern.7ex
  \char'47}}#1\relax\else\message{accent \string\soft \space #1 not
  defined!}#1\relax\fi\fi\fi\fi\fi\fi}
\begin{thebibliography}{71}
\providecommand{\natexlab}[1]{#1}
\providecommand{\msn}[1]{MR:\,\href{http://www.ams.org/mathscinet-getitem?mr=MR{#1}}{#1}}
\providecommand{\ZBL}[1]{Zbl:\,\href{https://www.zentralblatt-math.org/zmath/en/search/?q=an:#1}{#1}}
\providecommand{\url}[1]{\texttt{#1}}
\providecommand{\urlprefix}{}
\expandafter\ifx\csname urlstyle\endcsname\relax
  \providecommand{\doi}[1]{doi:\discretionary{}{}{}#1}\else
  \providecommand{\doi}{doi:\discretionary{}{}{}\begingroup
  \urlstyle{rm}\Url}\fi

\bibitem[{Alameddine(1980)}]{Alameddine80}
\textsc{Ahmad~Fawzi Alameddine}.
\newblock \href{https://doi.org/10.1002/jgt.3190040411}{On the number of cycles
  of length {$4$} in a maximal planar graph}.
\newblock \emph{J. Graph Theory}, 4(4):417--422, 1980.

\bibitem[{Alon and Caro(1984)}]{AC84}
\textsc{Noga Alon and Yair Caro}.
\newblock \href{https://doi.org/10.1016/S0304-0208(08)72803-2}{On the number of
  subgraphs of prescribed type of planar graphs with a given number of
  vertices}.
\newblock In \emph{Convexity and graph theory}, vol.~87 of \emph{North-Holland
  Math. Stud.}, pp. 25--36. North-Holland, 1984.

\bibitem[{Alon et~al.(2018)Alon, Kostochka, and Shikhelman}]{AKS18}
\textsc{Noga Alon, Alexandr Kostochka, and Clara Shikhelman}.
\newblock \href{https://doi.org/10.4310/joc.2018.v9.n4.a1}{Many cliques in
  {$H$}-free subgraphs of random graphs}.
\newblock \emph{J. Comb.}, 9(4):567--597, 2018.
\newblock \msn{3890929}.

\bibitem[{Alon and Shikhelman(2016)}]{AS16}
\textsc{Noga Alon and Clara Shikhelman}.
\newblock \href{https://doi.org/10.1016/j.jctb.2016.03.004}{Many {$T$} copies
  in {$H$}-free graphs}.
\newblock \emph{J. Combin. Theory Ser. B}, 121:146--172, 2016.
\newblock \msn{3548290}.

\bibitem[{Alon and Shikhelman(2019)}]{AS19}
\textsc{Noga Alon and Clara Shikhelman}.
\newblock \href{https://doi.org/10.1016/j.disc.2018.11.012}{{$H$}-free
  subgraphs of dense graphs maximizing the number of cliques and their
  blow-ups}.
\newblock \emph{Discrete Math.}, 342(4):988--996, 2019.
\newblock \msn{3894108}.

\bibitem[{Alweiss et~al.(pear)Alweiss, Lovett, Wu, and Zhang}]{ALWZ}
\textsc{Ryan Alweiss, Shachar Lovett, Kewen Wu, and Jiapeng Zhang}.
\newblock \href{http://arxiv.org/abs/1908.08483}{Improved bounds for the
  sunflower lemma}.
\newblock \emph{Annals of Math.}, to appear.
\newblock arXiv:1908.08483.

\bibitem[{Archdeacon(1986)}]{Archdeacon-JGT86}
\textsc{Dan Archdeacon}.
\newblock \href{https://doi.org/10.1002/jgt.3190100313}{The nonorientable genus
  is additive}.
\newblock \emph{J. Graph Theory}, 10(3):363--383, 1986.

\bibitem[{Barnette and Edelson(1988)}]{BE-IJM88}
\textsc{David~W. Barnette and Allan~L. Edelson}.
\newblock \href{https://doi.org/10.1007/BF02767355}{All orientable
  {$2$}-manifolds have finitely many minimal triangulations}.
\newblock \emph{Israel J. Math.}, 62(1):90--98, 1988.

\bibitem[{Barnette and Edelson(1989)}]{BE-IJM89}
\textsc{David~W. Barnette and Allan~L. Edelson}.
\newblock \href{https://doi.org/10.1007/BF02764905}{All {$2$}-manifolds have
  finitely many minimal triangulations}.
\newblock \emph{Israel J. Math.}, 67(1):123--128, 1989.

\bibitem[{Benjamini and Schramm(2001)}]{IO01}
\textsc{Itai Benjamini and Oded Schramm}.
\newblock \href{https://doi.org/10.1214/EJP.v6-96}{Recurrence of distributional
  limits of finite planar graphs}.
\newblock \emph{Electron. J. Probab.}, 6:no. 23, 2001.

\bibitem[{Borgs et~al.(2013)Borgs, Chayes, Kahn, and Lov{\'a}sz}]{BCKL13}
\textsc{Christian Borgs, Jennifer Chayes, Jeff Kahn, and L{\'a}szl{\'o}
  Lov{\'a}sz}.
\newblock \href{https://doi.org/10.1002/rsa.20414}{Left and right convergence
  of graphs with bounded degree}.
\newblock \emph{Random Structures Algorithms}, 42(1):1--28, 2013.

\bibitem[{Bubeck et~al.(2016)Bubeck, Edwards, Mania, and Supko}]{BEMS16}
\textsc{S\'ebastien Bubeck, Katherine Edwards, Horia Mania, and Cathryn Supko}.
\newblock \href{http://arxiv.org/abs/1601.01950}{On paths, stars and wyes in
  trees}.
\newblock 2016, arXiv:1601.01950.

\bibitem[{Bubeck and Linial(2016)}]{BL16}
\textsc{S\'{e}bastien Bubeck and Nati Linial}.
\newblock \href{https://doi.org/10.1002/jgt.21865}{On the local profiles of
  trees}.
\newblock \emph{J. Graph Theory}, 81(2):109--119, 2016.
\newblock \msn{3433633}.

\bibitem[{Chan et~al.(2021)Chan, Kráľ, Mohar, and Wood}]{CKMW21}
\textsc{Timothy F.~N. Chan, Daniel Kráľ, Bojan Mohar, and David~R. Wood}.
\newblock \href{http://arxiv.org/abs/2102.02010}{Inducibility and universality
  for trees}.
\newblock 2021, arXiv:2102.02010.

\bibitem[{Cheng et~al.(2004)Cheng, Dey, and Poon}]{CDP-CGTA04}
\textsc{Siu-Wing Cheng, Tamal~K. Dey, and Sheung-Hung Poon}.
\newblock \href{https://doi.org/10.1016/j.comgeo.2003.07.001}{Hierarchy of
  surface models and irreducible triangulations}.
\newblock \emph{Comput. Geom.}, 27(2):135--150, 2004.

\bibitem[{Czabarka et~al.(2017)Czabarka, Sz\'{e}kely, and Wagner}]{CSW17}
\textsc{\'{E}va Czabarka, L\'{a}szl\'{o} Sz\'{e}kely, and Stephan Wagner}.
\newblock \href{https://doi.org/10.37236/5943}{Paths vs. stars in the local
  profile of trees}.
\newblock \emph{Electron. J. Combin.}, 24(P1.22), 2017.
\newblock \msn{3609192}.

\bibitem[{Dujmovi\'{c} et~al.(2011)Dujmovi\'{c}, Fijav\v{z}, Joret, Sulanke,
  and Wood}]{DFJSW}
\textsc{Vida Dujmovi\'{c}, Ga\v{s}per Fijav\v{z}, Gwena\"el Joret, Thom
  Sulanke, and David~R. Wood}.
\newblock \href{https://doi.org/10.1016/j.ejc.2011.04.001}{On the maximum
  number of cliques in a graph embedded in a surface}.
\newblock \emph{European J. Combin.}, 32(8):1244--1252, 2011.

\bibitem[{Eppstein(1993)}]{Eppstein93}
\textsc{David Eppstein}.
\newblock \href{https://doi.org/10.1002/jgt.3190170314}{Connectivity, graph
  minors, and subgraph multiplicity}.
\newblock \emph{J. Graph Theory}, 17(3):409--416, 1993.

\bibitem[{Erd\H{o}s and Rado(1960)}]{ER60}
\textsc{Paul Erd\H{o}s and Richard Rado}.
\newblock \href{https://doi.org/10.1112/jlms/s1-35.1.85}{Intersection theorems
  for systems of sets}.
\newblock \emph{J. London Mathematical Society, Second Series}, 35(1):85--90,
  1960.

\bibitem[{Erd\H{o}s and Stone(1946)}]{ES46}
\textsc{Paul Erd\H{o}s and Arthur~H. Stone}.
\newblock \href{https://doi.org/10.1090/S0002-9904-1946-08715-7}{On the
  structure of linear graphs}.
\newblock \emph{Bull. Amer. Math. Soc.}, 52:1087--1091, 1946.
\newblock \msn{0018807}.

\bibitem[{Ergemlidze et~al.(2019)Ergemlidze, Methuku, Salia, and
  Gy\H{o}ri}]{EMSG19}
\textsc{Beka Ergemlidze, Abhishek Methuku, Nika Salia, and Ervin Gy\H{o}ri}.
\newblock \href{https://doi.org/10.1002/jgt.22390}{A note on the maximum number
  of triangles in a {$C_5$}-free graph}.
\newblock \emph{J. Graph Theory}, 90(3):227--230, 2019.
\newblock \msn{3904833}.

\bibitem[{Fomin et~al.(2010)Fomin, il~Oum, and Thilikos}]{FOT10}
\textsc{Fedor~V. Fomin, Sang il~Oum, and Dimitrios~M. Thilikos}.
\newblock \href{https://doi.org/10.1016/j.ejc.2010.05.003}{Rank-width and
  tree-width of {$H$}-minor-free graphs}.
\newblock \emph{European J. Combin.}, 31(7):1617--1628, 2010.

\bibitem[{Fox and Wei(2017)}]{FW17}
\textsc{Jacob Fox and Fan Wei}.
\newblock \href{https://doi.org/10.1016/j.jctb.2017.04.004}{On the number of
  cliques in graphs with a forbidden minor}.
\newblock \emph{J. Combin. Theory Ser. B}, 126:175--197, 2017.
\newblock \msn{3667668}.

\bibitem[{Gerbner et~al.(2018)Gerbner, Keszegh, Palmer, and
  Patk\'{o}s}]{GKPP18}
\textsc{D\'{a}niel Gerbner, Bal\'{a}zs Keszegh, Cory Palmer, and Bal\'{a}zs
  Patk\'{o}s}.
\newblock \href{https://doi.org/10.1137/16M109898X}{On the number of cycles in
  a graph with restricted cycle lengths}.
\newblock \emph{SIAM J. Discrete Math.}, 32(1):266--279, 2018.
\newblock \msn{3755658}.

\bibitem[{Gerbner et~al.(2019)Gerbner, Methuku, and Vizer}]{GMV19}
\textsc{D\'{a}niel Gerbner, Abhishek Methuku, and M\'{a}t\'{e} Vizer}.
\newblock \href{https://doi.org/10.1016/j.disc.2019.06.022}{Generalized
  {T}ur\'{a}n problems for disjoint copies of graphs}.
\newblock \emph{Discrete Math.}, 342(11):3130--3141, 2019.
\newblock \msn{3996750}.

\bibitem[{Gerbner et~al.(2020)Gerbner, Nagy, and Vizer}]{GNV20}
\textsc{D\'aniel Gerbner, Zolt\'an~L\'or\'ant Nagy, and M\'at\'e Vizer}.
\newblock \href{http://arxiv.org/abs/2008.12093}{Unified approach to the
  generalized {T}ur\'an problem and supersaturation}.
\newblock 2020, arXiv:2008.12093.

\bibitem[{Gerbner and Palmer(2019)}]{GP19}
\textsc{D\'{a}niel Gerbner and Cory Palmer}.
\newblock \href{https://doi.org/10.1016/j.ejc.2019.103001}{Counting copies of a
  fixed subgraph in {$F$}-free graphs}.
\newblock \emph{European J. Combin.}, 82:103001, 15, 2019.
\newblock \msn{3989658}.

\bibitem[{Ghosh et~al.(2021)Ghosh, Gy{\"{o}}ri, Martin, Paulos, Salia, Xiao,
  and Zamora}]{GGMPSXZ21}
\textsc{Debarun Ghosh, Ervin Gy{\"{o}}ri, Ryan~R. Martin, Addisu Paulos, Nika
  Salia, Chuanqi Xiao, and Oscar Zamora}.
\newblock \href{https://doi.org/10.1016/j.disc.2021.112317}{The maximum number
  of paths of length four in a planar graph}.
\newblock \emph{Discret. Math.}, 344(5):112317, 2021.

\bibitem[{Goodman(1959)}]{Goodman59}
\textsc{Adolph~W. Goodman}.
\newblock \href{https://doi.org/10.2307/2310464}{On sets of acquaintances and
  strangers at any party}.
\newblock \emph{Amer. Math. Monthly}, 66:778--783, 1959.

\bibitem[{Gurel-Gurevich and Nachmias(2013)}]{GN13}
\textsc{Ori Gurel-Gurevich and Asaf Nachmias}.
\newblock \href{https://doi.org/10.4007/annals.2013.177.2.10}{Recurrence of
  planar graph limits}.
\newblock \emph{Ann. of Math. (2)}, 177(2):761--781, 2013.

\bibitem[{Gy\H{o}ri et~al.(2019{\natexlab{a}})Gy\H{o}ri, Paulos, Salia,
  Tompkins, and Zamora}]{GPSTZb}
\textsc{Ervin Gy\H{o}ri, Addisu Paulos, Nika Salia, Casey Tompkins, and Oscar
  Zamora}.
\newblock \href{http://arxiv.org/abs/1909.13539}{The maximum number of paths of
  length three in a planar graph}.
\newblock 2019{\natexlab{a}}, arXiv:1909.13539.

\bibitem[{Gy\H{o}ri et~al.(2019{\natexlab{b}})Gy\H{o}ri, Paulos, Salia,
  Tompkins, and Zamora}]{GPSTZa}
\textsc{Ervin Gy\H{o}ri, Addisu Paulos, Nika Salia, Casey Tompkins, and Oscar
  Zamora}.
\newblock \href{http://arxiv.org/abs/1909.13532}{The maximum number of
  pentagons in a planar graph}.
\newblock 2019{\natexlab{b}}, arXiv:1909.13532.

\bibitem[{Gy\H{o}ri et~al.(2020)Gy\H{o}ri, Paulos, Salia, Tompkins, and
  Zamora}]{GPSTZc}
\textsc{Ervin Gy\H{o}ri, Addisu Paulos, Nika Salia, Casey Tompkins, and Oscar
  Zamora}.
\newblock \href{http://arxiv.org/abs/2002.04579v2}{Generalized planar {T}ur\'an
  numbers}.
\newblock 2020, arXiv:2002.04579v2.

\bibitem[{Gy\H{o}ri et~al.(2019{\natexlab{c}})Gy\H{o}ri, Salia, Tompkins, and
  Zamora}]{GSTZ19}
\textsc{Ervin Gy\H{o}ri, Nika Salia, Casey Tompkins, and Oscar Zamora}.
\newblock \href{https://dmtcs.episciences.org/5622}{The maximum number of
  {$P_\ell$} copies in {$P_k$}-free graphs}.
\newblock \emph{Discrete Math. Theor. Comput. Sci.}, 21(1):\#14,
  2019{\natexlab{c}}.
\newblock \msn{3981860}.

\bibitem[{Hakimi and Schmeichel(1979)}]{HS79}
\textsc{S.~Louis Hakimi and Edward~F. Schmeichel}.
\newblock \href{https://doi.org/10.1002/jgt.3190030108}{On the number of cycles
  of length {$k$} in a maximal planar graph}.
\newblock \emph{J. Graph Theory}, 3(1):69--86, 1979.

\bibitem[{Hakimi and Schmeichel(1982)}]{HS82}
\textsc{S.~Louis Hakimi and Edward~F. Schmeichel}.
\newblock \href{https://doi.org/10.1007/BF02760664}{Bounds on the number of
  cycles of length three in a planar graph}.
\newblock \emph{Israel J. Math.}, 41(1-2):161--180, 1982.

\bibitem[{Harant et~al.(2001)Harant, Horňák, and Skupień}]{HHS01}
\textsc{Jochen Harant, Mirko Horňák, and Zdzisław Skupień}.
\newblock \href{https://doi.org/10.1016/S0012-365X(01)00047-4}{Separating
  3-cycles in plane triangulations}.
\newblock \emph{Discrete Math.}, 239(1-3):127--136, 2001.
\newblock \msn{1850991}.

\bibitem[{Hatami et~al.(2014)Hatami, Lov{\'a}sz, and Szegedy}]{HLS14}
\textsc{Hamed Hatami, L{\'a}szl{\'o} Lov{\'a}sz, and Bal{\'a}zs Szegedy}.
\newblock \href{https://doi.org/10.1007/s00039-014-0258-7}{Limits of
  locally-globally convergent graph sequences}.
\newblock \emph{Geom. Funct. Anal.}, 24(1):269--296, 2014.

\bibitem[{Hatami and Norine(2011)}]{HN11}
\textsc{Hamed Hatami and Serguei Norine}.
\newblock \href{https://doi.org/10.1090/S0894-0347-2010-00687-X}{Undecidability
  of linear inequalities in graph homomorphism densities}.
\newblock \emph{J. Amer. Math. Soc.}, 24(2):547--565, 2011.

\bibitem[{Huynh and Wood(2021)}]{HW21}
\textsc{Tony Huynh and David~R. Wood}.
\newblock \href{https://doi.org/10.4153/s0008414x21000316}{Tree densities in
  sparse graph classes}.
\newblock \emph{Canadian J. Math.}, 2021.

\bibitem[{Joret and Wood(2010)}]{JoretWood-JCTB10}
\textsc{Gwena{\"e}l Joret and David~R. Wood}.
\newblock \href{https://doi.org/10.1016/j.jctb.2010.01.004}{Irreducible
  triangulations are small}.
\newblock \emph{J. Combin. Theory Ser. B}, 100(5):446--455, 2010.

\bibitem[{Lavrenchenko(1987)}]{Lavrenchenko}
\textsc{Serge Lavrenchenko}.
\newblock Irreducible triangulations of a torus.
\newblock \emph{Ukrain. Geom. Sb.}, 30:52--62, ii, 1987.
\newblock Translation in \emph{J. Soviet Math.} 51(5):2537--2543, 1990.

\bibitem[{Lawrencenko and Negami(1997)}]{LawNeg-JCTB97}
\textsc{Serge Lawrencenko and Seiya Negami}.
\newblock \href{https://doi.org/10.1006/jctb.1997.9999}{Irreducible
  triangulations of the {K}lein bottle}.
\newblock \emph{J. Combin. Theory Ser. B}, 70(2):265--291, 1997.

\bibitem[{Lee and Oum(2015)}]{LO15}
\textsc{Choongbum Lee and {Sang-il} Oum}.
\newblock \href{https://doi.org/10.1137/140979988}{Number of cliques in graphs
  with a forbidden subdivision}.
\newblock \emph{SIAM J. Disc. Math.}, 29(4):1999--2005, 2015.

\bibitem[{Liu(2021)}]{Liu21}
\textsc{Chun-Hung Liu}.
\newblock \href{http://arxiv.org/abs/2107.00874}{Homomorphism counts in
  robustly sparse graphs}.
\newblock 2021, arXiv:2107.00874.

\bibitem[{Lov{\'a}sz(2012)}]{Lovasz12}
\textsc{L{\'a}szl{\'o} Lov{\'a}sz}.
\newblock Large networks and graph limits.
\newblock Colloquium Publications, 2012.

\bibitem[{Lov{\'a}sz and Simonovits(1976)}]{LovSim76}
\textsc{L{\'a}szl{\'o} Lov{\'a}sz and Mikl{\'o}s Simonovits}.
\newblock On the number of complete subgraphs of a graph.
\newblock In \emph{Proc. of 5th British Combinatorial Conference}, vol.~XV of
  \emph{Congr. Numer.}, pp. 431--441. Utilitas Math., 1976.

\bibitem[{Lov{\'a}sz and Simonovits(1983)}]{LovSim83}
\textsc{L{\'a}szl{\'o} Lov{\'a}sz and Mikl{\'o}s Simonovits}.
\newblock \href{https://doi.org/10.1007/978-3-0348-5438-2_41}{On the number of
  complete subgraphs of a graph. {II}}.
\newblock In \emph{Studies in pure mathematics}, pp. 459--495. Birkh\"auser,
  1983.

\bibitem[{Luo(2018)}]{Luo18}
\textsc{Ruth Luo}.
\newblock \href{https://doi.org/10.1016/j.jctb.2017.08.005}{The maximum number
  of cliques in graphs without long cycles}.
\newblock \emph{J. Combin. Theory Ser. B}, 128:219--226, 2018.
\newblock \msn{3725194}.

\bibitem[{Ma and Qiu(2020)}]{MaQiu20}
\textsc{Jie Ma and Yu~Qiu}.
\newblock \href{https://doi.org/10.1016/j.ejc.2019.103026}{Some sharp results
  on the generalized {T}ur\'{a}n numbers}.
\newblock \emph{European J. Combin.}, 84:103026, 2020.
\newblock \msn{4014346}.

\bibitem[{Mantel(1907)}]{Mantel07}
\textsc{W.~Mantel}.
\newblock Problem 28.
\newblock \emph{Wiskundige Opgaven}, 10:60--61, 1907.

\bibitem[{Miller(1987)}]{Miller-JCTB87}
\textsc{Gary~L. Miller}.
\newblock \href{https://doi.org/10.1016/0095-8956(87)90028-1}{An additivity
  theorem for the genus of a graph}.
\newblock \emph{J. Combin. Theory Ser. B}, 43(1):25--47, 1987.

\bibitem[{Mohar and Thomassen(2001)}]{MoharThom}
\textsc{Bojan Mohar and Carsten Thomassen}.
\newblock Graphs on surfaces.
\newblock Johns Hopkins University Press, 2001.

\bibitem[{Nakamoto and Ota(1995)}]{NakaOta-JGT95}
\textsc{Atsuhiro Nakamoto and Katsuhiro Ota}.
\newblock \href{https://doi.org/10.1002/jgt.3190200211}{Note on irreducible
  triangulations of surfaces}.
\newblock \emph{J. Graph Theory}, 20(2):227--233, 1995.

\bibitem[{Ne{\v{s}}et{\v{r}}il and Ossona~de Mendez(2011)}]{NesOss11}
\textsc{Jaroslav Ne{\v{s}}et{\v{r}}il and Patrice Ossona~de Mendez}.
\newblock \href{https://doi.org/10.1016/j.ejc.2011.03.007}{How many {$F$}'s are
  there in {$G$}?}
\newblock \emph{European J. Combin.}, 32(7):1126--1141, 2011.

\bibitem[{Ne\v{s}et\v{r}il and Ossona~de Mendez(2020)}]{NO20}
\textsc{Jaroslav Ne\v{s}et\v{r}il and Patrice Ossona~de Mendez}.
\newblock \href{https://doi.org/10.1090/memo/1272}{A unified approach to
  structural limits and limits of graphs with bounded tree-depth}.
\newblock \emph{Mem. Amer. Math. Soc.}, 263(1272), 2020.
\newblock \msn{4069241}.

\bibitem[{Nikiforov(2011)}]{Nikiforov11}
\textsc{Vladimir Nikiforov}.
\newblock \href{https://doi.org/10.1090/S0002-9947-2010-05189-X}{The number of
  cliques in graphs of given order and size}.
\newblock \emph{Trans. Amer. Math. Soc.}, 363(3):1599--1618, 2011.

\bibitem[{Norine et~al.(2006)Norine, Seymour, Thomas, and Wollan}]{NSTW06}
\textsc{Serguei Norine, Paul Seymour, Robin Thomas, and Paul Wollan}.
\newblock \href{https://doi.org/10.1016/j.jctb.2006.01.006}{Proper minor-closed
  families are small}.
\newblock \emph{J. Combin. Theory Ser. B}, 96(5):754--757, 2006.

\bibitem[{Razborov(2008)}]{Razborov08}
\textsc{Alexander~A. Razborov}.
\newblock \href{https://doi.org/10.1017/S0963548308009085}{On the minimal
  density of triangles in graphs}.
\newblock \emph{Combin. Probab. Comput.}, 17(4):603--618, 2008.

\bibitem[{Reed and Wood(2009)}]{RW09}
\textsc{Bruce Reed and David~R. Wood}.
\newblock \href{https://doi.org/10.1145/1597036.1597043}{A linear time
  algorithm to find a separator in a graph excluding a minor}.
\newblock \emph{ACM Transactions on Algorithms}, 5(4):\#39, 2009.

\bibitem[{Reiher(2016)}]{Reiher16}
\textsc{Christian Reiher}.
\newblock \href{https://doi.org/10.4007/annals.2016.184.3.1}{The clique density
  theorem}.
\newblock \emph{Ann. of Math.}, 184(3):683--707, 2016.

\bibitem[{Ringel(1974)}]{Ringel74}
\textsc{Gerhard Ringel}.
\newblock Map color theorem.
\newblock Springer-Verlag, 1974.

\bibitem[{Sulanke(2006{\natexlab{a}})}]{Sulanke-Generating}
\textsc{Thom Sulanke}.
\newblock \href{http://arxiv.org/abs/0606687}{Generating irreducible
  triangulations of surfaces}.
\newblock arXiv:0606687, 2006{\natexlab{a}}.

\bibitem[{Sulanke(2006{\natexlab{b}})}]{Sulanke06}
\textsc{Thom Sulanke}.
\newblock \href{http://arxiv.org/abs/0606690}{Irreducible triangulations of low
  genus surfaces}.
\newblock arXiv:0606690, 2006{\natexlab{b}}.

\bibitem[{Sulanke(2006{\natexlab{c}})}]{Sulanke-KleinBottle}
\textsc{Thom Sulanke}.
\newblock \href{https://doi.org/10.1016/j.jctb.2006.05.001}{Note on the
  irreducible triangulations of the {K}lein bottle}.
\newblock \emph{J. Combin. Theory Ser. B}, 96(6):964--972, 2006{\natexlab{c}}.

\bibitem[{Timmons(2019)}]{Timmons19}
\textsc{Craig Timmons}.
\newblock
  \href{https://doi.org/10.1007/s00373-019-02054-x}{{$C_{2k}$}-saturated graphs
  with no short odd cycles}.
\newblock \emph{Graphs Combin.}, 35(5):1023--1034, 2019.
\newblock \msn{4003654}.

\bibitem[{Tur\'{a}n(1941)}]{Turan41}
\textsc{Paul Tur\'{a}n}.
\newblock On an extremal problem in graph theory.
\newblock \emph{Mat. Fiz. Lapok}, 48:436--452, 1941.

\bibitem[{Wood(2007)}]{Wood-GC07}
\textsc{David~R. Wood}.
\newblock \href{https://doi.org/10.1007/s00373-007-0738-8}{On the maximum
  number of cliques in a graph}.
\newblock \emph{Graphs Combin.}, 23(3):337--352, 2007.

\bibitem[{Wood(2016)}]{Wood16}
\textsc{David~R. Wood}.
\newblock \href{https://doi.org/10.37236/5715}{Cliques in graphs excluding a
  complete graph minor}.
\newblock \emph{Electronic J. Combinatorics}, 23(3):\#R18, 2016.
\newblock \msn{3558055}.

\bibitem[{Wormald(1986)}]{Wormald86}
\textsc{Nicholas~C. Wormald}.
\newblock \href{https://doi.org/10.1017/S0004972700010182}{On the frequency of
  {$3$}-connected subgraphs of planar graphs}.
\newblock \emph{Bull. Austral. Math. Soc.}, 34(2):309--317, 1986.

\bibitem[{Zykov(1949)}]{Zykov49}
\textsc{Alexander~A. Zykov}.
\newblock On some properties of linear complexes.
\newblock \emph{Mat. Sbornik N.S.}, 24(66):163--188, 1949.

\end{thebibliography}
\end{document}